\documentclass[review]{elsarticle}
\usepackage[utf8]{inputenc}
\usepackage{amsmath,amssymb} 
\usepackage{graphicx}
\usepackage{dsfont}
\usepackage{upgreek}
\usepackage{textcomp,comment}
\usepackage{braket,bigints}
\usepackage[margin=1in]{geometry}
\usepackage{amsthm}
\usepackage{mathrsfs}
\usepackage{mathtools}
\usepackage[table,svgnames]{xcolor}
\usepackage{graphicx}
\usepackage{float}

\usepackage[toc,page]{appendix}
\usepackage{enumitem}

\usepackage{tikz,soul}
\usetikzlibrary{arrows}
\tikzset{
  treenode/.style = {shape=rectangle, rounded corners,
                     draw, align=center,
                     top color=white, bottom color=blue!20},
  root/.style     = {treenode, font=\Large, bottom color=red!30},
  env/.style      = {treenode, font=\ttfamily\normalsize},
  dummy/.style    = {circle,draw}
}

\usetikzlibrary{calc,fadings,decorations.pathreplacing,arrows,
  datavisualization.formats.functions,shapes.geometric}
\usetikzlibrary{hobby}
\usetikzlibrary{positioning,calc}
\usetikzlibrary{arrows,decorations.markings}

\usepackage{cleveref}

\crefname{appsec}{appendix}{appendices}

\usepackage[mathlines]{lineno}

\newcommand{\trevor}[1]{\textcolor{orange}{#1}}
\newcommand{\cb}[1]{\textcolor{blue}{#1}}

\numberwithin{equation}{section}

\newtheoremstyle{newdefinition}
{20pt}
{20pt}
{}
{}
{\bfseries}
{.}
{.5em}
{}

\newtheorem{theorem}{Theorem}[section]
\newtheorem{lemma}{Lemma}[section]
\newtheorem{corollary}{Corollary}[section]
\newtheorem{proposition}{Proposition}[section]

\newtheorem{hypothesis}{Hypothesis}
\Crefname{hypothesis}{Hypothesis}{Hypotheses}
\crefname{hypothesis}{hypothesis}{hypotheses}

\theoremstyle{newdefinition}
\newtheorem{definition}{Definition}[section]

\theoremstyle{remark}
\newtheorem{remark}{Remark}[section]

\def\boxit#1{\vbox{\hrule\hbox{\vrule\kern6pt
			\vbox{\kern6pt#1\kern6pt}\kern6pt\vrule}\hrule}}

\def\bse{\begin{eqnarray*}}
	\def\ese{\end{eqnarray*}}
\def\be{\begin{eqnarray}}
	\def\ee{\end{eqnarray}}
\def\bq{\begin{equation}}
	\def\eq{\end{equation}}
\def\bse{\begin{eqnarray*}}
	\def\ese{\end{eqnarray*}}

\newcommand{\corb}[1]{\textcolor{blue}{#1}}

\newcommand{\bbR}{\mathbb{R}}

\newcommand{\bbN}{\mathbb{N}}

\newcommand{\bbC}{\mathbb{C}}

\newcommand{\bbP}{\mathbb{P}}

\newcommand{\bx}{\mathbf{x}}
\newcommand{\by}{\mathbf{y}}
\newcommand{\bv}{\mathbf{v}}
\newcommand{\bz}{\mathbf{z}}

\newcommand{\bbe}{\mathbf{e}}

\newcommand{\bphi}{\boldsymbol{\phi}}

\newcommand{\mcA}{{\mathcal A}}

\newcommand{\mcC}{{\mathcal C}}

\newcommand{\mcE}{{\mathcal E}}
\newcommand{\mcF}{\mathcal{F}}

\newcommand{\mcI}{{\mathcal I}}

\newcommand{\mcK}{{\mathcal K}}

\newcommand{\mcP}{{\mathcal P}}

\newcommand{\mcS}{\mathcal{S}}

\newcommand{\ii}{\mathbf{i}}

\newcommand{\pp}{\mathbf{p}}

\newcommand{\bbNset}{{\mathbb N}}

\newcommand{\eset}[1]{{\mathbb E} \left[ #1 \right] }

\newcommand{\Real}{\mathop{\text{\rm Re}}}
\newcommand{\Imag}{\mathop{\text{\rm Im}}}

\newcommand{\func}{u}

\definecolor{darkgreen}{rgb}{0, 0.6, 0}
\definecolor{airforceblue}{rgb}{0.36, 0.54, 0.66}
\definecolor{applegreen}{rgb}{0.55, 0.71, 0.0}
\definecolor{asparagus}{rgb}{0.53, 0.66, 0.42}
\definecolor{cadetblue}{rgb}{0.37, 0.62, 0.63}
\definecolor{cambridgeblue}{rgb}{0.64, 0.76, 0.68}
\definecolor{olivine}{rgb}{0.6, 0.73, 0.45}
\definecolor{rufous}{rgb}{0.66, 0.11, 0.03}
\definecolor{sangria}{rgb}{0.57, 0.0, 0.04}
\definecolor{neworange}{rgb}{1, 0.64, 0}
\definecolor{flowblue}{rgb}{0.4471,    0.6235,    0.8118}
\definecolor{lightsteelblue}{RGB}{176,196,222}
\definecolor{brownblue}{RGB}{222, 202, 176}

\theoremstyle{remark}

\begin{document}

\begin{frontmatter}

\journal{Computer and Mathematics with Applications}



\title{Analytic regularity of strong solutions for the complexified 
stochastic non-linear Poisson Boltzmann Equation}


\author[addr]{Brian Choi}
\ead{choigh@bu.edu}

\author[addr]{Jie Xu}
\ead{xujie@bu.edu} 

\author[addr]{Trevor Norton}
\ead{nortontm@bu.edu} 

\author[addr]{Mark Kon}
\ead{mkon@bu.edu}

\author[addr]{Julio E. Castrill\'on-Cand\'as \corref{mycorrespondingauthor}}
\ead{jcandas@bu.edu}

\tnotetext[t1]{This material is based upon work supported by the
  National Science Foundation (NSF) under Grant No. 1736392.  Research
  reported in this technical report was supported in part by the
  National Institute of General Medical Sciences (NIGMS) of the
  National Institutes of Health under award number 1R01GM131409-03. The first author was partially funded by the NSF/RTG postdoctoral fellowship DMS-1840260.}

\address[addr]{Department of Mathematics and Statistics, 
  Boston University, 665 Commonwealth Avenue, Boston, 02215, MA, USA}
 
\cortext[mycorrespondingauthor]{Corresponding author}

\begin{abstract}
Semi-linear elliptic Partial Differential Equations (PDEs) such as the non-linear Poisson Boltzmann Equation (nPBE) is highly relevant for non-linear electrostatics  in computational biology and chemistry. It is of particular importance for modeling  potential fields from molecules in  solvents or plasmas with stochastic fluctuations.  The extensive applications include ones in condensed matter and solid state physics, chemical physics, electrochemistry, biochemistry, thermodynamics, statistical mechanics, and materials science, among others.  In this paper we study the complex analytic properties of semi-linear elliptic Partial  Differential Equations with respect to random fluctuations on the domain. 
We first prove the existence and uniqueness of the nPBE on a bounded domain in $\mathbb{R}^3$. This proof relies on the application of a contraction mapping 
reasoning, as the standard convex optimization argument for the deterministic nPBE no longer applies.
Using the existence and uniqueness result we subsequently show that solution to the nPBE admits
an analytic extension onto a well defined region in the complex hyperplane with respect to the number 
of stochastic variables. Due to the analytic extension, stochastic collocation theory for
sparse grids predict algebraic to sub-exponential convergence rates with respect to the number of
knots. A series of numerical experiments with sparse grids is consistent with this prediction
and the analyticity result. Finally, this approach readily extends to a wide class of semi-linear elliptic PDEs.
\end{abstract}



\end{frontmatter}

\section{Introduction.}

Linear elliptic Partial Differential Equations (PDEs) have long been used to
model problems in chemistry, physics, engineering and biology
\cite{evans2010partial}. For example, the Poisson equation has been 
used extensively
in electrical engineering to model potential fields in semi-conductors, leading to the capacitance extraction problem. However,
as the physical scales are reduced, uncertainties in the shapes of 
semi-conductors leads to impedance matching problems \cite{Castrillon2016,Zhenhai2005}.
Moreover, many of these paradigms assume that the charge distribution (in
a dielectric) is contained in a vacuum. 
This is not the case for many important problems that lead to non-linear
electrostatics.

Non-linear elliptic PDEs have also been of much interest, with applications to  non-linear electrostatics, a subject that has gained much interest in modeling potential fields from molecules in solvents or plasma. Due to the interaction of 
the molecule and solvent (or plasma) a more accurate model is the 
non-linear elliptic PDE known as the non-linear Poisson Boltzmann Equation (nPBE). This
model is well known in Molecular Dynamics (MD) simulations
and chemical applications \cite{Gray2018,Stein2019}. As an example, it has been used
in modeling electrode-electrolyte interfaces \cite{Sundararaman2017,Sundararaman2018,Nattino2019}
and many related software packages have been developed \cite{Ringe2016,Jinnouchi2008,Mathew2014,Baker2001}.
As pointed out in \cite{Gray2018} the nPBE has been investigated in 
biophysics \cite{Ohshima2010,Phillips2009,Andelman1995,Benedek1979},
surface science \cite{Israelachvili1991,Butt2013}, chemical
physics \cite{Frenkel1946,Kirkwood1961}, polymer
physics \cite{Muthukumar2011}, plasma
physics \cite{Morrison1967,Ichimaru1984}, solid state
physics \cite{Ashcroft1976,Kittel1996}, condensed matter
physics \cite{Chaikin1995}, many-body
theory \cite{Brout1963,Giuliani2003}, thermodynamics
\cite{Glasstone1947,Lewis1961},
statistical 
mechanics \cite{Blum1992,Landau1958,McQuarrie1976},
liquid state theory 
\cite{Hansen2013,Fawcett2004,March1984}
, electrolyte solutions \cite{Barthel1998,Friedman1962},
electrochemistry \cite{Schmickler2010,Bockris1970,Sparnaay1972}, soft
matter \cite{Doi2016,Dean2014,Holm2001,Poon2006}, physical chemistry
\cite{Berry2000,Atkins1986}, biophysical chemistry \cite{Ohshima2010,Edsall1958},
biochemistry \cite{Bergethon1998}, medical physics
\cite{Hobbie1988}, physiology \cite{Bayliss1959,Hober1947}, molecular biology \cite{Sneppen2005}, colloids
\cite{Butt2013,Israelachvili1991,Verwey1948,Evans1999,Hunter2001,Lyklema1991},
applied mathematics \cite{Cai2013,Rubinstein1990,Li2009},
materials science \cite{Chavazaviel1999} and
technology \cite{McKelvey1966}.

The nPBE is given by
\begin{equation}\label{npb}
\begin{aligned}
	-\nabla \cdot (\epsilon(x) \nabla u)  + \kappa(x)^2 \sinh u &= f, & &x\in \Omega,\\
	u&= g, & &x \in \partial \Omega,
\end{aligned}
\end{equation}
where \(u\) is the nondimensionalized potential, \(\epsilon\) is the dielectric, and \( \kappa \) is the Debye-H\"uckel
parameter.  The nPBE has found important applications in protein interactions and molecular dynamics \cite{Padhorny2016,Neumaier1997}. In \Cref{nPBE:fig1} an example of the electrostatic potential field is rendered from the solution of the nPBE by using the Adaptive Poisson Boltzmann Solver (APBS) \cite{Baker2001}  for E. Coli RHo Protein. (PDB: 1A63 \cite{Berman2000}).  However, the mathematical properties of the nPBE are less understood and significantly more complicated than the linear case.

In practice may of these processes involve uncertainty in the dielectric material
and the geometry leading to uncertainty in the potential field, so that the solution of the
nPBE becomes stochastic. If we assume that uncertainty
in the domain is parameterized by $N$ random variables, then the solution of the nPBE is high dimensional 
and in many cases intractable.  However,
if sufficient complex analytic regularity of the solution with respect to the random variables
exists, then it can shown that with a stochastic collocation method and a sparse grid
polynomial representation, sub-exponential convergence is achieved
    \cite{nobile2008a,nobile2008b,Castrillon2016,Castrillon2021}.

\begin{figure}[htb]
	\centering
	\includegraphics[height = 8cm, width = 8cm, trim=6cm 6cm 6cm 6cm, clip]{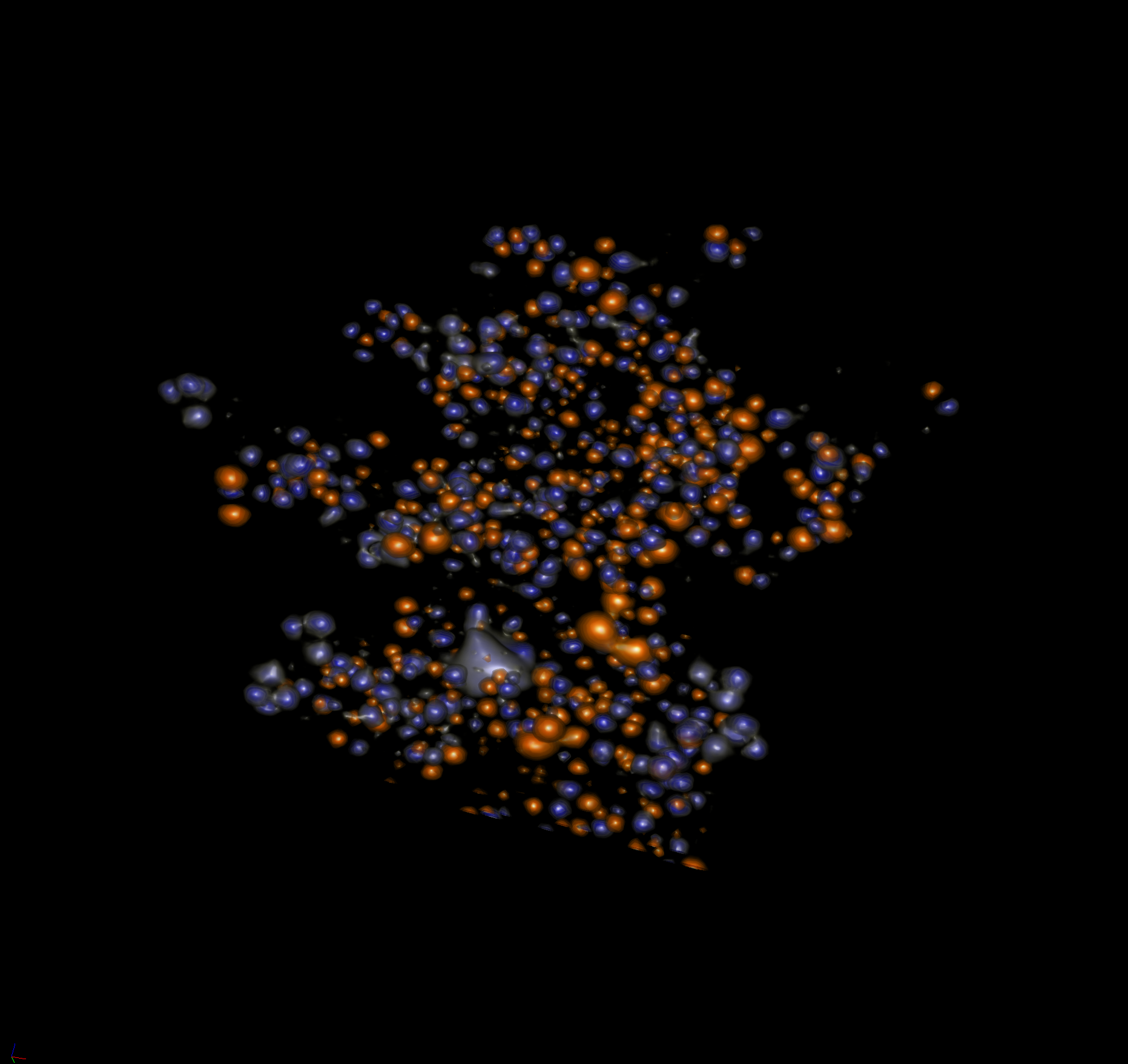}
	\caption{Electrostatic potential field obtained from the solution of the nPBE for the  
		RNA binding domain of E. Coli RHO factor. The potential fields were created with the Adaptive 
		Poisson Boltzmann Solver \cite{Baker2001} rendered with  VolRover \cite{Bajaj2003,Bajaj2005}.  The positive
		and negative potential are rendered with blueish and orange/reddish colors  respectively.}
	\label{nPBE:fig1}
\end{figure}

In \cite{Heitzinger2018} the authors investigate the application of 
stochastic collocation and stochastic Galerkin to the nPBE. In this paper the
nPBE is posed as semi-linear stochastic boundary valued problem and it
is proved that a unique solution exists, therefore extending the existence
and uniqueness result for the determistic case from  \cite{Holst1994}.
However, there are no analytic
regularity results, thus no convergence rates are derived in implementing
the stochastic collocation method.
In this paper the main contributions are:
\begin{itemize}
    \item \textbf{Proposition} \ref{Schauder}: We show existence and uniqueness of the complexified nPBE.
    \item \textbf{Theorem} \ref{mainresult2}; We also show that an analytic extension of the solution of the nPBE exists with respect to the random variables that describe the domain. 
\end{itemize}
The consequence of our results is that the stochastic nPBE can be solved with at least algebraic convergence
with respect to the dimensionality of the sparse grid.
Numerical results to study the convergence rates with respect to the dimensionality of the 
sparse grid are consistent with the analyticity results of \textbf{Theorem} \ref{mainresult2}.

This paper is organized as follows: Section \ref{Preliminaries} introduces notations and the 
fundamental mathematical background. In \Cref{exuniq}, the existence and uniqueness results are proved.
In \Cref{analyticreg} the existence and uniqueness results are used to show that 
 the solution to the nPBE admits an analytic extension in a well defined region in $\bbC^{N}$ 
In \Cref{polynomial} sparse grid tensor product polynomial approximations are discussed.
The convergence rates predicted from the error bounds for sparse grids are consistent with
the numerical experiments are performed in \Cref{numericalresults}. In \Cref{conclusion} final
remarks and future work are discussed. In addition,
In \Cref{constantest}, analytic estimates for the various constants used in this paper are shown
and in \Cref{appendixB}, a discussion on the failure of uniqueness of solutions is given.


\section{Preliminaries.}\label{Preliminaries}

\subsection{Fractional Sobolev Spaces}

For $k \in \mathbb{N}\cup \{0\}$ and $\Omega \subseteq \mathbb{R}^d$ open, bounded with a smooth boundary,
\[
	H^k(\Omega) = \{u \in L^2(\Omega): \lVert u \rVert_{H^k(\Omega)}<\infty\} \quad \text{with } \lVert u \rVert_{H^k(\Omega)} \coloneqq \Big(\sum_{|\alpha| \leq k} \lVert \partial_\alpha u \rVert_{L^2}^2\Big)^{\frac{1}{2}}.
\]
For $s^\prime \in [0,1)$, recall the Gagliardo seminorm
\[
	[u]_{s^\prime} \coloneqq \Big(\int_\Omega\int_\Omega \frac{|u(x)-u(y)|^2}{|x-y|^{d+2s^\prime}} dxdy\Big)^{\frac{1}{2}},
\]
by which fractional Sobolev spaces are defined. Given $s \geq 0$, we have $s = k + s^\prime$ for $k \in \mathbb{N}\cup\{0\}$ and $s^\prime \in [0,1)$. Define
\[
	H^s(\Omega) = \{u \in L^2(\Omega): \lVert u \rVert_{H^s} < \infty\}\quad \text{with } \lVert u \rVert_{H^s} \coloneqq \Big(\lVert u \rVert_{H^k}^2 + [u]_{s^\prime}^2 \Big)^{\frac{1}{2}},
\]
and $H^s_0(\Omega)$ to be the closure of $C^\infty_c(\Omega)$, the collection of smooth and compactly supported functions in $\Omega$, under $\lVert \cdot \rVert_{H^s}$. For a more thorough discussion on this material including the Sobolev spaces on the boundary $\partial \Omega$ and the negative-order Sobolev spaces, see \cite[Chapter 3]{mclean2000strongly}.

As for the regularity of boundary data, recall that $g:\partial \Omega \rightarrow \mathbb{C}$ and $w\in L^2(\Omega)$ such that $w=g$ in the trace sense. To invoke the elliptic regularity theorem, let $w \in H^2(\Omega)$. To motivate this assumption, consider the linear elliptic PDE $Lu=f$ on $\Omega$ with $u=g$ on $\partial \Omega$. Assuming that $w$ exists, a formal calculation reveals that $\tilde{u}+w$ is the solution where $\tilde{u}$ satisfies $L\tilde{u} = f-Lw$ on $\Omega$ with $\tilde{u}=0$ on $\partial \Omega$. If $f \in L^2(\Omega)$, we desire $Lw \in L^2(\Omega)$ to ensure that $\tilde{u}$ has two more derivatives than $f-Lw$. That $w \in H^2(\Omega)$ follows by assuming $g \in H^{\frac{3}{2}}(\partial \Omega)$.

\begin{lemma}\cite[Theorem 3.37]{mclean2000strongly}\label{trace}
	Let $T:C^\infty(\overline{\Omega})\rightarrow C^\infty (\partial \Omega)$ be given by $u \mapsto u|_{\partial \Omega}$. If $k \in \mathbb{N}$ and $\Omega \in C^{k-1,1}$, then $T$ uniquely extends to a surjective bounded linear operator from $H^s(\Omega)$ to $H^{s-\frac{1}{2}}(\partial\Omega)$ for all $s \in (\frac{1}{2},k]$. The trace map $T$ has a right-continuous inverse. 
\end{lemma}
The map $g\mapsto w$ is not unique although one can uniquely solve the Laplace equation
\begin{equation}\label{poisson}
\begin{split}
	\Delta w&=0,\: x \in \Omega\\
	w&=g,\ x \in \partial \Omega,
\end{split}    
\end{equation}
and obtain an explicit form for $w$ as an integration against the Poisson kernel, and thereby establish a map $T^{-1}g\coloneqq w$, the definition of inverse trace adopted in this paper. It can be shown that $T^{-1}: H^{k-\frac{1}{2}}(\partial \Omega)\rightarrow H^k(\Omega)$ defines a bounded linear operator where the operator norm depends on $k\in \mathbb{N}$ and $\Omega$. 

\subsection{Principal Eigenvalue of the Dirichlet Laplacian}

Let $\Omega \subseteq \mathbb{R}^d$ be open, bounded, and convex with a smooth boundary. Let $|\Omega|$ denote the Lebesgue measure of $\Omega$ and $d_\Omega \coloneqq \sup_{x,y\in \Omega}|x-y|$, the diameter of $\Omega$. Let $\left\{\lambda_i\right\}_{i=1}^\infty$ be the eigenvalues of the (negative) Dirichlet Laplacian $-\Delta$, the Laplacian operator restricted to functions vanishing on the boundary defined via the Friedrich extension with the linear ordering
\begin{equation*}
0<\lambda_1<\lambda_2\leq \lambda_3\leq \dots    
\end{equation*}

The principal eigenvalue is given by
\[
	\lambda_1 = \min_{u \in H^1_0(\Omega)\setminus \{0\}} \frac{\int_{\Omega} |\nabla u|^2}{\int_\Omega |u|^2}.    
\]
The variational formula above directly implies the Poincar\'e inequality given below:
\begin{equation}\label{poincare2}
	\lVert u \rVert_{L^2(\Omega)}^2\leq \lambda_1^{-1}\lVert \nabla u \rVert_{L^2(\Omega)}^2, \quad \forall u \in H^1_0(\Omega).
\end{equation}

\section{Existence and uniqueness.}\label{exuniq}

\subsection{Problem Set-up}\label{sketch_assumptions}

Re-expressing \eqref{npb} in the form 
\begin{equation*}
    F(u) := Lu + N(u) = f,
\end{equation*}
where $f \in L^2(\Omega)$ and $F:H^2(\Omega)\rightarrow L^2(\Omega)$ is decomposed into the linear and nonlinear parts, the desired solution $u$ is realized as a fixed point to the map
\begin{equation*}
    u \mapsto L^{-1}( f - N(u)).
\end{equation*}

Appropriate restrictions are imposed on the given parameters to account for the possible exponential growth of \(\| N(u) \|\). 

\begin{hypothesis}\label{h1}
	The domain \(\Omega\subset \mathbb R^d\) is open, bounded, and convex with smooth (at least \(C^2\)) boundary.
\end{hypothesis}
\begin{hypothesis}\label{h2}
	The function \(\epsilon \in W^{1,\infty}(\Omega, \mathbb C^{d^2})\)
	, where \(\epsilon^{ij} = \epsilon^{ji}\) for \(1 \leq i,j \leq d\), satisfies the following uniform ellipticity condition: there exists \(\theta > 0\) such that 
	\begin{equation}\label{ellipticity}
		\Real\left[ \sum_{i,j=1}^d \epsilon^{ij}(x) \xi_i \overline{\xi_j} \right] \geq \theta |\xi|^2
	\end{equation}
	for all \(\xi\in \mathbb{C}^d\) and for a.e.\ \(x\in \Omega\).
\end{hypothesis}
\begin{hypothesis}\label{h3}
	There exists \(\mu \geq 0\) such that \(\kappa^2 \in L^\infty(\Omega)\) satisfies 
	\begin{equation}\label{kappa_bound}
		\Real\left[\mathrm{\kappa^2(x)}\right] \geq -\mu
	\end{equation}
	for a.e.\ \(x\in \Omega\). Moreover, the parameters \(\theta\) and \(\mu\) characterized by \Cref{ellipticity,kappa_bound}, respectively, satisfy the inequality 
	\begin{equation}
		\frac\mu\theta < \lambda_1,
	\end{equation}
	where \(\lambda_1\) is the principal eigenvalue of \(-\Delta\) on \(\Omega\).
\end{hypothesis}

\begin{remark}
\Cref{h2,h3} are sufficient to guarantee a unique strong solution within a small ball in \(H^2(\Omega)\), and omitting these hypotheses may result in non-uniqueness. If the parameters are allowed to continuously vary until \Cref{h2,h3} no longer hold, then there may be a bifurcation of the unique small solution. See \cref{appendixB} for details.
\end{remark}


\subsection{Definitions}

\begin{definition}\label{npbweak}
	A function $u \in H^1(\Omega)$ is a \emph{weak solution} to \Cref{npb} if for all $\phi \in H^1_0(\Omega)$, we have
	\begin{equation}
	\begin{split}
	\label{npbweak2}
		\int_\Omega (\epsilon\nabla u) \cdot \overline{\nabla \phi} + \int_\Omega \kappa^2\sinh u \cdot\overline{\phi} &=\int_\Omega f \overline{\phi}\\
		u|_{\partial\Omega}&=g,
		\end{split}
	\end{equation}
	where the equality at the boundary is in the trace sense; if $w \in H^1(\Omega)$ whose trace is $g \in H^{\frac{1}{2}}(\partial \Omega)$, a weak solution $u$ satisfies $u-w \in H^1_0(\Omega)$. If a weak solution $u$ is twice weakly-differentiable and satisfies \Cref{npb} pointwise almost everywhere, then we say $u$ is a \emph{strong solution}. 
\end{definition}

A weak solution that is in $H^2(\Omega)$ satisfies the strong form a.e.\ if one can undo the integration by parts in the first term of \Cref{npbweak2}. This is possible since \(\epsilon\in W^{1,\infty}(\Omega, \mathbb{C}^{d^2})\) . If $u \in H^2(\Omega)$ is a weak solution, then one can undo the integration by parts since $\epsilon \nabla u \in H^1(\Omega)$. Indeed for each $1 \leq i,k \leq d$,
\[
	\lVert \partial_k \sum_{j = 1}^{d} (\epsilon^{ij}\partial_j u)\rVert_{L^2} \leq C \lVert \epsilon \rVert_{W^{1,\infty}}\lVert u \rVert_{H^2}.
\]

Setting up the functional equation given in \Cref{sketch_assumptions}, the non-linear operator \(F:H^2(\Omega)\to L^2(\Omega)\) is given by
\begin{equation}
\begin{split}
F(u) &= -\nabla\cdot(\epsilon \nabla u) + \kappa^2 \sinh(u)\\
Lu &= -\nabla\cdot(\epsilon\nabla u) + \kappa^2 u \label{set-up}\\
N(u) &= \kappa^2\sinh(u) = \kappa^2 \sum_{k=2}^\infty n_{k} u^{k},    
\end{split}
\end{equation}
where $L$ is a uniformly elliptic second-order differential operator by \Cref{h2} and $n_{k} = \frac{1}{k!}$ for odd \(k\) and $n_k = 0$ for even \(k\). Note that when \(d\leq 3\), \(H^2(\Omega)\) is an algebra and so \(u^{k}\in H^2(\Omega)\) and \(N(u)\) is well-defined. 

To simplify the
presentation, the relevant operator norms are defined. 
Explicit estimates for these constants and their  dependencies are derived in detail in Appendix A.

\begin{definition}\label{important_constants}
The Sobolev embedding for $s>\frac{d}{2}$ gives that \(H^s(\Omega) \hookrightarrow L^\infty(\Omega)\) \cite[Theorem 3.26]{mclean2000strongly}. The constant $C_S = C_S(s,\Omega,d)>0$ is then defined as the norm of the inclusion operator from \(H^s(\Omega)\) into \(L^\infty(\Omega)\):
\begin{equation*}
    C_S:= \inf\{ C>0 : \|u\|_{L^\infty(\Omega)} \leq C \|u\|_{H^s(\Omega)},\ \forall u \in H^s(\Omega)\}.
\end{equation*}
The linear operator $L$ in \Cref{set-up} takes functions in \(H^2(\Omega)\cap H^1_0(\Omega)\) to functions in \(L^2(\Omega)\) We shall denote \(C_D>0\) to be the norm of \(L\) with respect to these spaces:
\begin{equation*}
    C_D := \inf\{ C>0 : \|Lu\|_{L^2(\Omega)} \leq C \|u\|_{H^2(\Omega)},\ \forall u \in H^2(\Omega)\}.
\end{equation*}
By \Cref{h3}, \(L\) is invertible and so \(L^{-1}:L^2(\Omega) \to H^2(\Omega) \cap H^1_0(\Omega)\) is defined. Define \(C_H>0\) to be the norm of \(L^{-1}\):
\begin{equation*}
    C_H := \inf\{ C>0 : \|L^{-1} u\|_{H^2(\Omega)} \leq C \|u\|_{L^2(\Omega)},\ \forall u \in L^2(\Omega)\}.
\end{equation*}
\end{definition}
\subsection{Main Results}

Recall the well-posedness of the linear PBE.
\begin{lemma}\label{laxmilgram}
	Let $L$ be as in \Cref{set-up} and let $f\in L^2(\Omega)$. Fix $w \in H^2(\Omega)$ whose trace is $g \in H^{\frac{3}{2}}(\partial\Omega)$. Then, there exists a unique $v \in H^1(\Omega)$ such that $v = g$ in the trace sense and
	\[
		\int_\Omega (\epsilon \nabla v)\cdot \overline{\nabla\phi} + \kappa^2 v \overline{\phi} = \int_\Omega f\overline{\phi},
	\]
	for all $\phi \in H^1_0(\Omega)$. 
\end{lemma}

Let $L$ be given by \Cref{set-up} and consider solving $Lu=f \in L^2(\Omega)$ with $g=0$ on $\partial \Omega$. From the standard elliptic theory (for instance, see \cite[Section 6.3, Theorem 4]{evans2010partial}), there exists $C_H>0$ as defined in \Cref{important_constants} such that $\lVert u \rVert_{H^2} \leq C_H \lVert f \rVert_{L^2}$. A general boundary datum can be absorbed into the homogeneous term by replacing $f$ by $f - Lw$ and considering the zero Dirichlet boundary condition, which yields

\begin{lemma}\label{ellipticregularity}
	The unique weak solution of the linear PBE, $u \in H^1(\Omega)$, satisfies
	\begin{equation}\label{ellipticregularity2}
		\lVert u \rVert_{H^2} \leq C_H \lVert f \rVert_{L^2} + (C_HC_D + 1)\lVert w \rVert_{H^2}.
	\end{equation}
\end{lemma}

The nonlinear term is handled iteratively where, at each iteration, the regularity gain coming from the elliptic regularity theory is used to estimate the non-linear term in $L^2(\Omega)$.

\begin{lemma}
	Let $s>\frac{d}{2}$. Then for every $v\in H^s(\Omega)$,
	\begin{equation}\label{gn}
		\lVert N(v) \rVert_{L^2} \leq \| \kappa \|_{L^\infty}^2|\Omega|^\frac{1}{2}\sum_{k=2}^\infty |n_k|(C_S(s) \lVert v \rVert_{H^s})^k.  
	\end{equation}
	For $N(v)$ for \Cref{npb}, we have
	\[
		\lVert N(v)\rVert_{L^2} \leq \| \kappa \|_{L^\infty}^2|\Omega|^{\frac{1}{2}}\Big(\sinh(C_S(s) \lVert v \rVert_{H^s}) - C_S(s) \lVert v \rVert_{H^s}\Big).
	\]
\end{lemma}

\begin{proof}
	\[
		\lVert N(v)\rVert_{L^2} \leq \| \kappa \|_{L^\infty}^2\sum_{k=2}^\infty |n_k| \lVert v^k \rVert_{L^2} \leq \| \kappa \|_{L^\infty}^2|\Omega|^{\frac{1}{2}}\sum_{k=2}^\infty |n_k| \lVert v\rVert_{L^\infty}^k\leq \| \kappa \|_{L^\infty}^2|\Omega|^\frac{1}{2}\sum_{k=2}^\infty |n_k|(C_S(s) \lVert v \rVert_{H^s})^k, 
	\]
	where the inequalities are by the triangle inequality, the H\"older's inequality, and the Sobolev embedding, respectively.
\end{proof}

\begin{remark}
	By working in Sobolev algebras, we bypass the problem of whether $N(u) \in L^2(\Omega)$ or not, for $u \in L^\infty(\Omega)$. Note that Holst \cite[Chapter 2]{Holst1994} bypasses this issue as well, not by working with more regular functions as we do, but by constructing a conditional action functional on $H^1_0(\Omega)$. For $d=3$, the embedding $H^2(\Omega)\hookrightarrow L^\infty(\Omega)$ holds, but an analogous embedding for $H^1(\Omega)$ does not hold. For $d \geq 3$, it is straightforward to construct an example of $u \in H^1(\Omega)$ such that $N(u) = \kappa^2 (\sinh u -u) \notin L^2(\Omega)$. For simplicity, take $\Omega = \mathbb{R}^d$ and $\kappa=1$. For $R>0$, define $u(x) = |x|^{-\alpha}\zeta(x)$ where $\alpha = \frac{d}{2}-1-\epsilon$ with $\epsilon \ll 1$ and $\zeta \in C^\infty_c(B(0,R))$ is a smooth non-negative function such that $\zeta = 1$ on $\overline{B(0,\frac{R}{2})}$. If $N(u) \in L^2(\Omega)$, then $\sinh (u(\cdot)) \in L^2(\Omega)$. Since $\sinh u \geq \frac{u^N}{N!}$ for every odd $N \geq 1$, we have $u^N \in L^2(\Omega)$. However, this is false due to the blow-up of $u$ at the origin.
\end{remark}

Define a non-linear operator $A: C^\infty_c(\Omega)\rightarrow H^2(\Omega)$ where for every $v \in C^\infty_c(\Omega)$, $A(v)$ satisfies
\begin{equation}
\label{nonlinearoperator}
\begin{split}
	L(A(v)) &= f-N(v),\:x\in \Omega\\
	A(v)&= g,\: x \in \partial \Omega.
	\end{split}
\end{equation}
Denoting $K: L^2(\Omega)\rightarrow H^1_0(\Omega)\cap H^2(\Omega)$ to be the inverse of $L$ stated in \Cref{ellipticregularity}, conclude
\[
	A(v) = K(f-N(v)-Lw)+w. 
\]

Equivalently, the operator $A$ defines an iteration map on some Banach space where each iterate is a unique solution to the linear PBE. More precisely, let $u_0 \in H^2(\Omega)$ be the solution for the linearized nPBE; by \Cref{laxmilgram}, there exists a unique weak solution $u_0 \in H^1(\Omega)$ and by \Cref{ellipticregularity}, $u_0 \in H^2(\Omega)$. By the Sobolev embedding theorem, $H^2(\Omega)\hookrightarrow C^{0,\frac{1}{2}}(\overline{\Omega})$, and therefore $N(u_0)\in L^2(\Omega)$. Then, consider
\begin{equation}\label{approxsol}
\begin{split}
	Lu_k + N(u_{k-1}) &= f,\: k \geq 1,\\
	u_k &= g,
	\end{split}
\end{equation}
or equivalently, $u_k = A(u_{k-1})$. It may be that the sequence is convergent to a function that does not solve \eqref{npb} in any meaningful way; see \cite{brezis2007nonlinear}. By showing that $A$ uniquely extends to a compact operator on fractional Sobolev spaces, it is shown that $\left\{u_k\right\}$ converges to a solution of \eqref{npb}. In the rest of this section, assume $d=3$ unless specified otherwise.
\begin{proposition}\label{nonlinearcompact}
	$A:L^\infty(\Omega)\rightarrow H^2(\Omega)$ is locally Lipschitz continuous. Consequently, $A$ is continuous on $H^s(\Omega)$ for every $s \in (\frac{3}{2},2]$. Furthermore, $A$ is compact on $H^s(\Omega)$ for every $s\in (\frac{3}{2},2)$.
\end{proposition}
\begin{proof}
Let $u,v \in L^\infty(\Omega)$ with $\| u \|_{L^\infty}, \| v \|_{L^\infty(\Omega)} \leq M$. Then
from the Mean Value Theorem,
\begin{equation}\label{continuityest}
\begin{split}
\lVert A(u)-A(v)\rVert_{H^2} &= \lVert K(N(u)-N(v))\rVert_{H^2} \leq C_H \lVert N(u)-N(v)\rVert_{L^2}\\
&\leq C_H 
\| u-v \|_{L^\infty} \int_0^1 \| N^\prime((1-t)u+tv) \|_{L^2} dt\\
&\leq C_H \| \kappa \|_{L^\infty}^2\| u-v \|_{L^\infty} \int_0^1 \sum_{k \geq 2} k |n_k| \| (1-t)u +tv \|_{L^{2(k-1)}}^{k-1}dt\\
&\leq \Big\{C_H \| \kappa \|_{L^\infty}^2|\Omega|^{\frac{1}{2}} \sum_{k \geq 2} k |n_k| 2^{k-2} (\| u \|_{L^\infty}^{k-1}+\| u \|_{L^\infty}^{k-1})\Big\}\| u-v \|_{L^\infty}.
\end{split}
\end{equation}
If we show that the series in \Cref{continuityest} converges, then the proof is complete. By the Cauchy integral formula, we obtain an upper bound on $|n_k|$
	\[
		|n_k| \leq \frac{\max\limits_{|z|=R}|N(z)|}{R^{k}},
	\]
	and combining this bound for $R = 4M$ with \Cref{continuityest}, the infinite sum is a convergent geometric series, and therefore the desired continuity has been shown. 
	
	The rest follows from the Sobolev embedding and the Rellich-Kondrachov compactness theorem. Indeed, the continuity of $A$ on $H^s(\Omega)$ follows by considering the embedding $H^s(\Omega)\hookrightarrow L^\infty(\Omega)$. Compactness of $A$ on $H^s(\Omega)$ is immediate once we show $A$ sends a bounded subset of $H^s(\Omega)$ into a bounded subset of $H^2(\Omega)$. For $\lVert u \rVert_{H^s}\leq M$,
	\begin{equation}\label{brouwer0}
	\begin{split}
		\lVert A(u)\rVert_{H^2} &= \lVert K(f-N(u)-Lw)+w\rVert_{H^2} \leq C_H \lVert f-N(u)-Lw \rVert_{L^2}+\lVert w \rVert_{H^2} \\
		&\leq C_H \lVert f \rVert_{L^2} + (C_HC_D+1)\lVert w \rVert_{H^2} + C_H \lVert N(u)\rVert_{L^2} \\
		&\leq C_H \lVert f \rVert_{L^2} + (C_HC_D+1)\lVert w \rVert_{H^2} + C_H \| \kappa \|_{L^\infty}^2  |\Omega|^\frac{1}{2}\sum_{k=2}^\infty |n_k|(C_S(s) M)^k, 
	\end{split}
	\end{equation}
	where \Cref{brouwer0} follows from \Cref{gn}.
\end{proof}

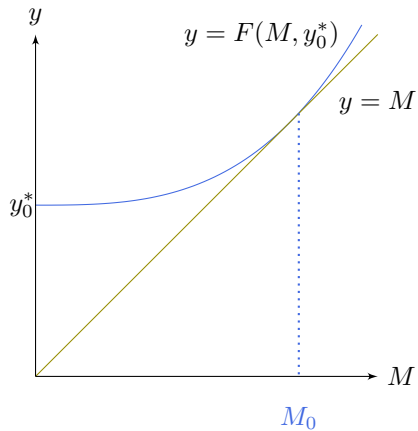
\begin{figure}[ht]
		\centering
		\newcounter{j} %
		\begin{tikzpicture}  
			[>=latex',scale=7,
			declare function={%
				p(\t)= greater(\t,0.5)  ? 1 : 
				0.3256 + sinh( 2 * \t ) - 2 * \t;} ]
			\draw[color=RoyalBlue,samples at={0,0.01,...,0.62}] plot (\x,{0.3256 + sinh( 2 * \x ) - 2 * \x});  
			\draw[color=olive](0,0)--(0.65,0.65);
			\draw[->](0,0)--(0,0.65) node[above]{$y$};
			\draw[->](0,0)--(0.65,0) node[right]{$M$};
			

			\draw[color=RoyalBlue,dotted,line width=0.8pt]%
			(0.5,0.5)--(0.5,0) node[below=8pt]{$M_0$};
			
			\node at (0.43,0.65) {$y = F(M,y_0^{\ast})$};
			\node at (0.65,0.52) {$y = M$};
			\node at (-0.025,0.3256) {$y_0^\ast$};
		\end{tikzpicture}
		\caption{The tangency condition of \Cref{small} where $y_0 = C_H \lVert f \rVert_{L^2} + (C_HC_D+1)\lVert w \rVert_{H^2}$ and $F$ is the LHS of \Cref{small}}\label{figure}
\end{figure} 

Now we apply the a priori estimate above to obtain a strong solution to \Cref{npb}.

\begin{proposition}\label{Schauder}
Let $s\in (\frac{3}{2},2)$ and $M>0$ satisfy	\begin{equation}\label{Schauder2}
C_H \lVert f \rVert_{L^2} + (C_HC_D+1)\lVert w \rVert_{H^2} + C_H \| \kappa \|_{L^\infty}^2  |\Omega|^\frac{1}{2}\sum_{k=2}^\infty |n_k|(C_S(s) M)^k\leq M.   
\end{equation}
Then there exists a strong solution $u \in H^2(\Omega)$ to \eqref{npb} with $\lVert u \rVert_{H^s} \leq M$. If the condition
\begin{equation}\label{Banach2}
C_H C_S(s) \| \kappa \|_{L^\infty}^2|\Omega|^{\frac{1}{2}} \sum_{k=2}^\infty k |n_k| (C_S(s) M)^{k-1}<1, 
\end{equation}
 holds in addition to \eqref{Schauder2}, then $u \in H^2(\Omega)$ is unique in $\overline{B_{H^s}(0,M)}$.
\end{proposition}
\begin{proof}
Let $\lVert u \rVert_{H^s} \leq M$. Then, a similar argument to \eqref{brouwer0} yields $\lVert A(u)\rVert_{H^s} \leq M$. By Schauder's fixed point theorem (\cite[Corollary 11.2]{gilbarg2015elliptic}) on $\overline{B_{H^s}(0,M)}$, a closed convex subset of $H^s(\Omega)$ on which $A$ is continuous and compact by \Cref{nonlinearcompact}, we obtain $u \in \overline{B_{H^s}(0,M)}$ such that $A(u)=u$. Since $A$ is smoothing, $u \in H^2(\Omega)$.

Following the steps leading to \Cref{continuityest}, we obtain
\begin{equation}\label{Banach3}
\lVert A(u)-A(v)\rVert_{H^s} \leq \left(C_H C_S(s) \| \kappa \|_{L^\infty}^2|\Omega|^{\frac{1}{2}} \sum_{k=2}^\infty k |n_k| (C_S M)^{k-1}\right)\lVert u-v \rVert_{H^s},
\end{equation}
where the strict contraction of $A$ on $\overline{B_{H^s}(0,M)}$ follows from \eqref{Banach2} and the Banach fixed point theorem. 
\end{proof}


\begin{remark}\label{small_data}
    When $N(u) = \sinh u$, there exists $M>0$ that satisfies \eqref{Schauder2} if and only if
    \begin{equation*}
    y_0 := C_H\| f \|_{L^2} + (C_H C_D + 1)\| w \|_{H^2}    
    \end{equation*}
    is sufficiently small such that
\begin{equation}\label{small}
    y_0 + C_H \| \kappa \|_{L^\infty}^2  |\Omega|^{\frac{1}{2}}(\sinh C_S M_0 - C_S M_0)\leq M_0
\end{equation}
where
\begin{equation}\label{solution_bound}
M_0 = C_S^{-1} \cosh^{-1}(1+\frac{1}{C_H C_S \| \kappa \|_{L^\infty}^2  |\Omega|^{1/2}}),    
\end{equation}
which is independent of $\| f\|_{L^2}, \|w \|_{H^2}$. Hence the maximum value of $y_0$ that satisfies \eqref{small} in the equality is
\begin{equation}\label{small_data2}
\begin{split}
y_0^{\ast} &= M_0 - C_H \| \kappa \|_{L^\infty}^2  |\Omega|^{\frac{1}{2}}(\sinh C_S M_0 - C_S M_0)\\
&= C_S^{-1} \left( (1+C_H C_S \| \kappa \|_{L^\infty}^2  |\Omega|^{1/2})\cosh^{-1}\left( 1+\frac{1}{C_H C_S \| \kappa \|_{L^\infty}^2  |\Omega|^{1/2}}\right) - \sqrt{1+2C_H C_S \| \kappa \|_{L^\infty}^2  |\Omega|^{1/2}}\right),
\end{split}    
\end{equation}
which can be shown to be strictly positive. Note that there exists $C>0$ such that
\begin{equation*}
    y_0^\ast (\| \kappa \|_\infty) \sim C_S^{-1} \ln \left( \frac{C}{\lVert \kappa \rVert_{L^\infty}^2  }\right)\ \text{as}\ \| \kappa \|_\infty \rightarrow 0,
\end{equation*}
with other parameters fixed, which is consistent with the global existence and uniqueness of \eqref{npb} with $\kappa = 0$ where \eqref{npb} reduces to the linear PBE.

Given that \eqref{small} holds and $y_0 \leq y_0^{\ast}$, the uniqueness condition \eqref{Banach2} is equivalent to $M<M_0$. Therefore if $y_0 < y_0^{\ast}$, there exists a strong solution $u \in H^2(\Omega)$ unique in $\{ \| u \|_{H^s} < M_0 \}$ by choosing $M = M_0 - \epsilon$ for $\epsilon > 0$ sufficiently small. On the other hand, uniqueness is not guaranteed if $y_0 = y_0^{\ast}$.

The equality case of \eqref{small} when $y_0 = y_0^\ast$ is illustrated in \Cref{figure}. The curve $M \mapsto F(\cdot,y_0^\ast)$ tangentially intersects the identity.
\end{remark}
\begin{remark}
    As can be shown in \eqref{Schauder2}, \eqref{Banach2} our fixed point approach works for small data where the given parameters must be small measured in various norms. For complexified nPBE, however, this restriction is necessary if we wish to preserve uniqueness of solution; see \Cref{nonunique}. On the other hand, our approach establishes existence and uniqueness for a wide class of nonlinearities that may grow super-linearly.
\end{remark}

In the spirit of \Cref{nonlinearcompact,Schauder}, the (local) Lipschitz continuity of $A$ on Sobolev spaces can be shown with more restrictive hypotheses. Denote $A \lesssim B$ whenever there exists an implicit constant $C>0$ dependent only on the given parameters such that $A \leq CB$.

\begin{corollary}
Let $s>\frac{d}{2}$ and $\sigma \leq s \leq \sigma+2$ for $\sigma \in \mathbb{N}\cup \{0\}$, and assume $\epsilon \in C^{\sigma+1}(\overline{\Omega}),\kappa \in H^s(\Omega)$ and $\partial \Omega$ is $C^{\sigma+2}$. For any given $f \in H^\sigma(\Omega), g \in H^{\sigma+\frac{3}{2}}(\partial \Omega)$, it follows that $A:H^s(\Omega)\rightarrow H^{\sigma+2}(\Omega)$ is locally Lipschitz continuous. Hence for $s> \frac{d}{2}$ and $\sigma \leq s < \sigma+2$, there exists a unique fixed point for $A$ in $B_{H^s(\Omega)}(0,R)\subseteq H^s(\Omega)$ for some $R>0$, if the smallness hypotheses analogous to \Cref{Banach2,Banach3} are assumed. 
\end{corollary}
\begin{proof}
By the Elliptic Regularity Theorem and the Mean Value Theorem,
    \begin{equation}
\begin{split}
\lVert A(u)-A(v)\rVert_{H^{\sigma+2}} &= \lVert K(N(u)-N(v))\rVert_{H^{\sigma+2}} \lesssim \lVert N(u)-N(v)\rVert_{H^\sigma}\\
&\lesssim \| \kappa \|_{H^s}^2\| u-v \|_{H^s} \int_0^1 \| N^\prime((1-t)u+tv) \|_{H^s} dt.
\end{split}
\end{equation}
The rest is analogous to the proofs of \Cref{nonlinearcompact,Schauder} using the Sobolev algebra property and the fixed point arguments.
\end{proof}

\section{Analytic Regularity.}
\label{analyticreg}

In this section, the complex analytic regularity of the unique solution obtained in \Cref{Schauder} is shown. In the spirit of uncertainty quantification, the coefficient functions $\bv := (\epsilon, \kappa, f, g)$ are assumed to depend on $\by := (y_1,\dots,y_N) \in \Gamma = [-1,1]^N$ where $\Gamma$ denotes the stochastic domain. Suppose the coefficient functions admit analytic extensions to $\Theta \subseteq \mathbb{C}^N$ containing $\Gamma$; the notation $\bz := (z_1,\dots,z_N) \in \Theta$ is used to denote the complexification of $\by$. For convenience, absorb the boundary data and consider
\begin{equation}\label{npb2}
\begin{aligned}
	-\nabla_x \cdot \left(\epsilon(x,\by) \nabla_x v(x,\by)\right)  + \kappa(x,\by)^2 \sinh \left(v(x,\by)+w(x,\by)\right) &= f(x,\by) + \nabla_x \cdot (\epsilon(x,\by)\nabla_x w(x,\by)),  &&x\in \Omega,\\
	v(x,\by)&= 0,  &&x \in \partial \Omega,
\end{aligned}
\end{equation}
where $\by \in \Gamma$ and $u= v + w$ is the solution to \eqref{npb}.

\begin{hypothesis} (\textbf{Finite Dimensional Noise Model}) 
We assume that that fluctuations of the components of $\bv$ are modeled as
\[
\bv(x,\by):= \xi_{\Omega}(x) + \sum_{i=0}^{N} \alpha_i \bphi(x) y_i
\]
where $\alpha_i \in R$, $\xi_{\Omega}$ is the indicator function defined on $\Omega$  and $\bphi_i \in [C^{\infty}(\Omega)]^4 \cap
[L^{2}(\Omega)]^4$ for $i = 1,\dots,N$.
\end{hypothesis}

\begin{remark}
Suppose we have a complete probability space $(\Omega_{\bbP},\mcF,\bbP)$, where $\Omega_{\bbP}$ is the set of outcomes, 
$\mcF$ is the sigma algebra and $\bbP$ is the associated probability measure. The stochastic fluctuations are modeled by $\Gamma = [-1,1]^N$ reflects the finite-dimensional noise assumption that the coefficient functions are random inputs described by $N$ real-valued random variables $\{ Y_k(w_p)\}_{k=1}^N$ for $w_p \in \Omega_{\bbP}$ with mean zero and unit variance where the domain $\Gamma$ may also be unbounded as described in \cite{babusk_nobile_temp_10}. Hence $\mathbf{v}(x,w_p) := \begin{bmatrix} \epsilon(x,w_p) & \kappa(x,w_p) & f(x,w_p) & g(x,w_p) \end{bmatrix}^T \in L^2_{\bbP}(\Omega;[L^{2}(\Omega)]^4)$ can be represented by the expansion
\begin{equation*}
\mathbf{v}(x,w_p) = \mathbb{E}[\mathbf{v}(x,w_p)] + \sum_{k=1}^N \sqrt{\lambda_k} \pmb{\phi}_k(x) Y_k(w_p),\ \lambda_k \geq 0.    
\end{equation*}

\noindent If the covariance structure of $\mathbf{v}$ were known, then $(\lambda_k,\pmb{\phi}_k)_{k=1}^N$ forms the eigenbasis of the covariance matrix and therefore by orthogonality (see \cite[Section 2]{castrillon2022stochastic} for an introduction to the vector-field Karhunen-Lo\`eve transform),
\begin{equation*}
    Y_k(w_p) = \frac{1}{\sqrt{\lambda_k}} \int_{\Omega} (\mathbf{v}(x,w_p) - \mathbb{E}[\mathbf{v}(x,w_p)])^T \pmb{\phi}_k(x) dx.
\end{equation*}    
\end{remark}

\begin{remark}
    Due to the KL truncation the coefficients $\epsilon$
    and $\kappa$ can become negative i.e. the truncation does not guarantee that the 
    expansion of $\epsilon$ and $\kappa$ are always positive and thus a solution might 
    not exist. This problem is circumvented by taking to the KL expansion of $\bv$ and
    replacing $\epsilon$ and $\kappa$ with the
    terms $\log (\epsilon - a_{min})$, $\min\{\epsilon\} > a_{min} > 0$ and $\log (\kappa)$
    respectively. An
    exponential is then applied to the truncated KL expansion of $\bv$. This will guarantee that $\epsilon$ and $\kappa$ will always
    be positive.
\end{remark}

Furthermore the analytic extension of $u(\by)$ from $y \in \Gamma$ to an open subset $\Theta \subseteq \mathbb{C}^N$ containing a Bernstein polyellipse $\mathcal{E}_{\rho_1}\times \cdots \times \mathcal{E}_{\rho_N} \subseteq \Theta$ containing $\Gamma$ with $\rho_i >1,\ 1 \leq i \leq N$ implies a algebraic or sub-exponential convergence of the surrogate model to the exact solution with respect to the number of collocation points used in the sparse grid approximation as shown in \cite[Lemma 4.4]{babusk_nobile_temp_10}.

For the next proposition, let $\epsilon(z) := \epsilon(\cdot,z) \in W^{1,\infty}(\Omega;\mathbb{C})$ be a scalar field for $z \in \mathbb{C}$ satisfying $\Real(\epsilon(x,\bz)) \geq \theta>0$ for all $x \in \Omega$. 

\begin{theorem}\label{mainresult2}
Suppose $\epsilon,\kappa,f,w$ admit analytic extensions from $\Gamma$ to an open subset $\Theta \subseteq \mathbb{C}^N$ such that for all $\bz \in \Theta$, the following quantitative bounds hold
\begin{equation}\label{quantitative}
\begin{split}
0< d_\Omega^2 &< \pi^2 \theta C_H,\\
C_H \| f(\bz) \|_{L^2} + (C_H C_D + 1) \| w(\bz) \|_{H^2} &< y_0^\ast(\bz),\\
\| \kappa(\bz) \|_{L^\infty}^2 &\leq \frac{\pi^2 \theta}{d_\Omega^2} - \frac{1}{C_H},
\end{split}    
\end{equation}
with $y_0^\ast$ defined as \eqref{small_data2}. Then the solution map $\Theta \rightarrow H^2(\Omega)$ given by $\bz \mapsto u(\bz)$ is complex-analytic.
\end{theorem}
\begin{remark}
The upper bound on $d_\Omega>0$ and the size restriction on $\Theta$ are sufficient conditions in \Cref{mainresult2}, but need not be necessary. To elaborate on the smallness of $d_\Omega>0$, we have 
\begin{equation*}
    C_H \geq C_D^{-1} \geq \frac{1}{18 \| \epsilon \|_{W^{1,\infty}} + \| \kappa \|_{L^\infty}^2},
\end{equation*}
by \Cref{important_constants} and \eqref{bounded}. Therfore if $d_\Omega^2 < \frac{\pi^2 \theta}{18 \| \epsilon \|_{W^{1,\infty}} + \| \kappa \|_{L^\infty}^2}$, then the assumption $d_\Omega^2 < \pi^2 \theta C_H$ is satisfied.
\end{remark}

\begin{proof}

A strategy of the proof is given. By the Hartog's Theorem \cite[Chap 1]{krantz2001function}, any separately holomorphic functions are continuous on $\Theta$, and therefore analytic on the Bernstein polyellipse by the Osgood's Lemma \cite[Chap 1]{gunning2022analytic}. Let $\bz = (z_n;z_n^\prime) \in \Theta$ where $z_n$ is the $n^{\text{th}}$ coordinate of $\bz$ and $z_n^\prime$ is the rest of the coordinates. Let $s := \Re z_n,\ \omega := \Im z_n$. Based on the proof of \cite[Lemma 8]{Castrillon2021}, the existence of the partial derivatives of the solution ($\partial_s u,\partial_\omega u$) is justified, after which, $u(z)$ is shown to satisfy the Cauchy-Riemann equation in $z_n$ by taking the complex derivatives of \eqref{npb2}. 

By \eqref{quantitative}, the nPBE \eqref{npb2} admits a solution $v:\Theta \rightarrow H^1_0(\Omega) \cap H^2(\Omega)$ unique in $B_{H^s(\Omega)}(0,M_0)$ by \Cref{small_data}. Let $\zeta_1,\zeta_2$ be the formal derivatives $\partial_s v_R,\partial_s v_I$, respectively. Taking the formal partial derivatives of \eqref{npb2} in $s$, the vector equation that consists of the real and imaginary parts is given by 
\begin{equation}\label{partial}
\begin{split}
    &-\nabla \cdot
    \left(\begin{pmatrix}
    \epsilon_R & -\epsilon_I\\
    \epsilon_I & \epsilon_R
    \end{pmatrix}
    \nabla
    \begin{pmatrix}
    \zeta_1\\ \zeta_2
    \end{pmatrix}\right)+
    \mathbf{V}
    \begin{pmatrix}
    \zeta_1\\ \zeta_2
    \end{pmatrix}=
    \begin{pmatrix}
        f_1 \\ f_2
    \end{pmatrix},
\end{split}
\end{equation}
where

\begin{equation*}
    \mathbf{V}(z) = \begin{pmatrix}
    \cosh(u_R)\cos(u_I) & -\sinh(u_R) \sin(u_I)\\
    \sinh(u_R) \sin(u_I) & \cosh(u_R)\cos(u_I) 
    \end{pmatrix}
    \begin{pmatrix}
    \kappa_R^2 - \kappa_I^2 & -2\kappa_R \kappa_I\\
    2\kappa_R \kappa_I & \kappa_R^2 - \kappa_I^2
    \end{pmatrix},
\end{equation*}
and $f_1(\bz),f_2(\bz) \in L^2(\Omega)$. Considering $L^2(\Omega), H^1_0(\Omega)$ over $\mathbb{R}$, the product $\mathscr{H} = H^1_0(\Omega) \times H^1_0(\Omega)$ defines a Hilbert space under

\begin{equation*}
\langle \phi,\psi \rangle = \left\langle
    \begin{pmatrix}
        \phi_1 \\ \phi_2
    \end{pmatrix}
    ,
    \begin{pmatrix}
        \psi_1 \\ \psi_2
    \end{pmatrix}
    \right\rangle = \int_\Omega \nabla \phi_1 \cdot \nabla \psi_1 + \nabla \phi_2 \cdot \nabla \psi_2.
\end{equation*}

Let $B:\mathscr{H} \times \mathscr{H} \rightarrow \mathbb{R}$ be the bilinear form corresponding to \eqref{partial}. Consider $B = B_1 + B_2$ defined by

\begin{equation*}
\begin{split}
B_1(\phi,\psi) &= \bigintsss_\Omega \left( \begin{pmatrix}
    \epsilon_R & -\epsilon_I\\
    \epsilon_I & \epsilon_R
    \end{pmatrix}
    \begin{pmatrix}
        \nabla \phi_1 \\ \nabla \phi_2
    \end{pmatrix}\right)
    \cdot
    \begin{pmatrix}
        \nabla \psi_1 \\ \nabla \psi_2
    \end{pmatrix}dx\\
    B_2(\phi,\psi) &= \bigintsss_\Omega \left(\mathbf{V}
    \begin{pmatrix}
    \phi_1 \\ \phi_2
    \end{pmatrix}\right)
    \cdot
    \begin{pmatrix}
     \psi_1 \\ \psi_2
    \end{pmatrix}dx.    
\end{split}
\end{equation*}

By the uniform ellipticity condition, we have

\begin{equation*}
    |B_1(\phi,\psi)| \leq \| \epsilon(\bz) \|_{L^\infty} \| \phi \|_{\mathscr{H}}\| \psi \|_{\mathscr{H}},\ B_1(\phi,\phi) \geq \theta \| \phi \|_{\mathscr{H}}^2.
\end{equation*}
Another standard exercise yields

\begin{equation*}
    |B_2(\phi,\psi)| \leq \left(2\lambda_1^{-1} \cosh(2\| u_R(\bz)\|_{L^\infty})\|\kappa(\bz)\|_{L^\infty}^2\right) \| \phi \|_{\mathscr{H}} \| \psi \|_{\mathscr{H}}.
\end{equation*}

Let $r(x,\bz) := \sqrt{\cosh^2(u_R)|\kappa|^4 - (2\kappa_R\kappa_I)^2}$.
Combining

\begin{equation*}
\begin{split}
B_2(\phi,\phi) &= \int_\Omega (\cosh (u_R) \cos (u_I) (\kappa_R^2 - \kappa_I^2) - 2 \sinh(u_R) \sin (u_I)\kappa_R \kappa_I)(\phi_1^2 + \phi_2^2)\\
& \geq -\int_\Omega r(x,\bz)(\phi_1^2 + \phi_2^2) \geq -\| r(\cdot,\bz) \|_{\infty} \lambda_1^{-1} \| \phi \|_{\mathscr{H}}^2,
\end{split}
\end{equation*}
with the coercivity estimate of $B_1$, we have
\begin{equation}\label{B_coercive}
    B(\phi,\phi) \geq \left(\theta - \frac{\| r(\cdot,\bz)\|_{\infty}}{\lambda_1}\right)\| \phi \|_{\mathscr{H}}^2.
\end{equation}

By the Sobolev Embedding Theorem and \eqref{solution_bound}, we obtain

\begin{equation*}
\begin{split}
    r(x,\bz) &\leq \cosh (u) |\kappa(\bz)|^2 \leq \cosh (C_S(s) \| u \|_{H^s}) |\kappa(\bz)|^2 < \cosh (C_S(s)M_0) |\kappa(\bz)|^2\\
    & \leq \| \kappa(\bz) \|_{L^\infty}^2 + \frac{1}{C_H C_S(s) |\Omega|^{\frac{1}{2}}},
\end{split}
\end{equation*}
where $s \in (\frac{3}{2},2)$. Since $C_S(s) \geq C_S(2)$, it follows from \Cref{sobolevestimate} that $C_S(s) |\Omega|^{\frac{1}{2}} \geq 1$. By \eqref{quantitative} and the lower bound, $\lambda_1 \geq \frac{\pi^2}{d_\Omega^2}$, on the principal eigenvalue of the Dirichlet Laplacian (see \cite{payne1960optimal}),
\begin{equation*}
    r(x,\bz) < \| \kappa(\bz) \|_{L^\infty}^2 + \frac{1}{C_H} \leq \frac{\pi^2 \theta}{d_\Omega^2} \leq \lambda_1 \theta,
\end{equation*}
and hence the coercivity of $B$ from \eqref{B_coercive}. By the Lax-Milgram Theorem, \eqref{partial} is well-posed in $\mathscr{H}$ and by the elliptic regularity theory, it follows that $\begin{pmatrix}
\zeta_1(\bz) \\ \zeta_2(\bz)
\end{pmatrix} \in \mathscr{H} \cap (H^2(\Omega)\times H^2(\Omega))$. Adopting the argument in \cite[Lemma 8]{Castrillon2021} where \eqref{npb2} is considered in its weak form, it can be shown that $\partial_s v_R(\bz),\partial_s v_I(\bz)$ exist for every $z \in \Theta$ and

\begin{equation*}
    \zeta_1(\bz) = \partial_s v_R(\bz),\ \zeta_2(\bz) = \partial_s v_I(\bz),    
\end{equation*}
while taking the formal derivative of \eqref{npb2} in the imaginary variable yields an analogous conclusion for $\partial_\omega v$. 

Let $P(x,\bz) = \partial_s v_R(\bz) - \partial_\omega v_I(\bz)$ and $Q(x,\bz) = \partial_\omega v_R(\bz) + \partial_s v_I(\bz)$. By taking the partial derivatives of \eqref{npb2} and using the Cauchy-Riemann equations satisfied by the coefficient functions given by the hypotheses, one can argue as (22) of \cite{Castrillon2016} that $P,Q$ satisfy \eqref{partial} with $\zeta_1 = P,\ \zeta_2 = Q$ and $f_1 = f_2 = 0$ under the zero Dirichlet boundary condition. By the Lax-Milgram Theorem, $P(z_n;z_n^\prime) = Q(z_n;z_n^\prime) = 0$ for all

\begin{equation*}
z_n \in \Theta(z_n^\prime):=\{ \zeta \in \mathbb{C}: (\zeta;z_n^\prime) \in \Theta\},  
\end{equation*}
for each $z_n^\prime$. Since $v:\Theta \rightarrow H^1_0(\Omega) \cap H^2(\Omega)$ defines a complex-analytic map, so does $u = v + w$ define a $H^2(\Omega)$-valued analytic map. Indeed the analytic extension of $g$ implies that of $w = T^{-1}g$.
\end{proof}



\section{Tensor sparse grid polynomial approximation}
\label{polynomial}

Consider the problem of approximating a function $\nu: \Gamma \rightarrow W$
on the domain $\Gamma$ and Banach space $W$. The goal is to efficiently approximate such a function
using tensor product polynomials. The accuracy of the polymonial representation
will depend on the existence of an analytic extension of $\nu$ onto the
complex domain $\Theta \subset \bbC^{N}$. In this section we introduce
a brief discussion of sparse grids, which have become a popular method
for approximating functions with complex analytic extensions. A more detail
exposition can be found in \cite{babusk_nobile_temp_10,nobile2008a,Castrillon2016,Castrillon2021}.

Let $\mcP_{ p}(\Tilde{\Gamma}):=\text{\rm span}(y^k,\,k=0,\dots,p)$ and  
consider the univariate function $\nu: \Tilde{\Gamma} \rightarrow W$, $\Tilde{\Gamma} = [-1,1]$  and
univariate Lagrange interpolant 
$\mcI^{m(i)}_{p}:C^{0}(\Tilde{\Gamma}) \rightarrow {\mcP}_{m(i)-1}(\Tilde{\Gamma})$
such that

\[
\mcI^{m(i)}_{p}
(\nu(y)):= \sum_{k=1}^{m(i)} \nu(\cdot, {y_k}) l^{p}_{k}(y),
\]
where $i \geq 0$ is the level of approximation and
$m(i) \in \bbN_{0}$ is the number of evaluation points at level $i \in
\bbN_{0}$,  $m(0) = 0$, $m(1) = 1$ and $m(i) \leq m(i+1)$ if $i \geq
1$. By convention, for $m(0)=0$, let $\mcP_{-1} =\emptyset$. Furthermore, 
the Lagrange polynomials $l^p_k$ form  a basis for $\mcP_p(\tilde \Gamma)$.

The Lagrange interpolant can be easily extended to $N$ dimensions by taking
tensor products. Let $\mcP_{\pp}(\Gamma) = \bigotimes_{n=1}^{N}\;\mcP_{p_n}(\Gamma_{n})$
be an index of polynomial degree along each dimension of $\Gamma$. The interpolant 
$\mcI_{\pp}:C^{0}(\Gamma) \rightarrow \mcP_{\pp}(\Gamma)$ can be
constructed as $\mcI^{\pp}:= \mcI^{m(i)}_{p_1} \otimes \mcI^{m(i)}_{p_2} \otimes \dots \otimes 
\mcI^{m(i)}_{p_N}$. Thus for any function $\nu: \Gamma \rightarrow W$ we have
that
\[
\mcI^{\pp}(\nu(\by)) = \sum_{k \in \mcK} \nu(\cdot, {\by_k}) l^{\pp}_{k}(\by),
\]
where $\mcK$ is an appropriate index set and $l^{\pp}_{k}(\by)$ are 
tensors of univariate Lagrange polynomials. However, the
dimensionality of $\mcP_p$ and the index set $\mcK$  increases as 
$\prod_{n=1}^N$ $(p_n+1)$  with $N$. This polynomial representation 
suffers from the curse of dimensionality and 
becomes intractable for even a moderate amount of dimensions \cite{Khoromskij2018}.
In contrast, if sufficient regularity of $\nu(\by)$ exists with respect to
$\by \in \Gamma$ the dimensionality of the polynomial approximation can
be significantly reduced by restricting the degree of the tensor product polynomial
along each dimension.

Consider the difference operator along the {$n^{th}$ dimension of
  $\Gamma$}
\[
  {\Delta_{p,n}^{m(i)} :=} \mcI^{m(i)}_{p}-\mcI^{m(i-1)}_{p}.
  \]
Let $\ii=(i_1,\ldots,i_{N}) \in \bbNset^{N}_{+}$ be a multi-index,
$w$ be the level of approximation of the sparse grid
and 
$g:\bbNset^{N}_{+}\rightarrow\bbNset$ be a restriction function
that acts on $\ii$. The sparse grid approximation of $\nu$ is constructed as
\[
  \mcS^{m,g}_w[\nu]
= \sum_{\ii\in\bbNset^{N}_{+}: g(\ii)\leq w} \;\;
 \bigotimes_{n=1}^{N} {\Delta_n^{m(i_n)}}(\nu(\by)). 
\]
By choosing the function $m(i)$ and $g$ appropriately the dimensionality
of the sparse grid can be controlled.  The Smolyak sparse grid \cite{nobile2008a} 
can be built with the following formulas:
\[
m(i) = \begin{cases} 1, & \text{for } i=1 \\ 2^{i-1}+1, & \text{for }
  i>1 \end{cases}\quad \text{ and } \quad g(\ii) = \sum_{n=1}^N
(i_n-1),
\]
where
\[
    f(p) = \begin{cases}
      0, \; p=0 \\
      1, \; p=1 \\
      \lceil \log_2(p) \rceil, \; p\geq 2
    \end{cases}.
\]
The second step is to choose the location of the knots $\by_k$. For the domain
$\Gamma$ a good
choice is the Clenshaw-Curtis (CC) abscissa \cite{nobile_tempone_08}. This
consists of the extrema of Chebyshev polynomials:
\[
y^n_j = -cos \left( \frac{\pi(j-1)}{m(i) - 1} \right),
\]

\begin{remark}
Note that isotropic sparse grids can be extended to anisotropic
setting. By adapting the function $g$ to the contribution of each dimension 
to the function $\nu(\by)$, higher convergence rates can be 
achieved \cite{nobile2008b}.
\end{remark}

We are now interested in obtaining bounds for the error of the
sparse grid i.e. $\|\nu - \mcS^{m,g}_{v}[\nu]
\|_{L^{\infty}(\Gamma)}$. This bound can be shown to be  controlled
by i) the number of dimensions $N$,
ii) the number of knots $\eta$ in the sparse grid, but most
importantly iii) the size of the region of complex region of the
analytic extension of $\nu(\by)$ onto $\bbC^{N}$.

Let ${\mathcal E}_{\hat \sigma_1, \dots, \hat \sigma_{N}} : = \Pi_{n=1}^{N}{\mathcal E}_{n,\hat \sigma_n} \subset \bbC^{N}$, where 
\[
\begin{split}
  \mcE_{n,\hat \sigma_n}
  &= \Big\{ z \in \bbC \mathrm{\ with}\,\Real{z} =
  \frac{e^{\delta_n} + e^{-\delta_n} }{2}\cos(\theta),\,\,\,\Imag{z} =
  \frac{e^{\delta_n} - e^{-\delta_n}}{2}\sin(\theta): \theta \in
       [0,2\pi) , \hat \sigma_n \geq \delta_{n} \geq 0 \Big\}
  \end{split}
\]
and $\hat \sigma_n > 0$.  It is known that if a function $u:\tilde \Gamma \rightarrow \bbR$ 
can be  analytically extended onto the Bernstein ellipse ${\mathcal E}_{n,\hat \sigma_n}$ then the
error of the Lagrange  interpolation of $u$ will decay  exponentially with respect to the parameter
$\hat \sigma_n > 0 $ \cite{Castrillon2021}. This result has been extended to sparse grid polynomials in \cite{nobile2008a,nobile2008b} on suitable Banach spaces. However, the converge rate can
was shown to be algebraic or sub-exponential.

We set $\hat \sigma \equiv \min_{n=1,\dots,N} \hat \sigma_n,$ i.e. the decay is going to 
controlled by the Bernstein ellipse with the smallest size. This setting corresponds
to an isotropic sparse grid.
Suppose that $W = \bbR$ and  let
\[
\tilde{M}(\nu) := \sup_{\bz \in \mcE_{\hat \sigma_1, \dots, \hat
    \sigma_{N}}} |\nu(\bz)|.
\]
In addition let 
$\sigma = \hat{\sigma}/2$, $\mu_1 = \frac{\sigma}{1 + \log (2N)}$, and
$\mu_2(N) = \frac{\log(2)}{N(1 + \log(2N))}$, $\tilde{C}_{2}(\sigma) = 1 + \frac{1}{\log{2}}\sqrt{
\frac{\pi}{2\sigma}}
,\,\,\delta^{*}(\sigma) = \frac{e\log{(2)} - 1}{\tilde{C}_2
  (\sigma)}, 
C_1(\sigma,\delta,\tilde M(\func)) = \frac{4\tilde{M}(\func) C(\sigma)a(\delta,\sigma)}{ e\delta\sigma},$
\[
a(\delta,\sigma):=
\exp{
\left(
\delta \sigma \left\{
\frac{1}{\sigma \log^{2}{(2)}}
+ \frac{1}{\log{(2)}\sqrt{2 \sigma}}
+ 2\left( 1 + \frac{1}{\log{(2)}} 
\sqrt{ \frac{\pi}{2\sigma} }
\right)
\right\}
\right)
},
\]
$\mu_3 = \frac{\sigma \delta^{*} \tilde C_2(\sigma)}{1 + 2 \log
  (2N)}$, and
\[{\mathcal Q}(\sigma,\delta^{*}(\sigma),N, \tilde M(\func)) = 
\frac{ C_1(\sigma,\delta^{*}(\sigma),\tilde M(\func))}{\exp(\sigma
  \delta^{*}(\sigma) \tilde{C}_2(\sigma) )}
\frac{\max\{1,C_1(\sigma,\delta^{*}(\sigma),\tilde M(\func))\}^{N}}{|1
  - C_1(\sigma,\delta^{*}(\sigma),\tilde M(\func))|}.
\]

The following results is obtained \cite{nobile2008a,Castrillon2021,Castrillon2021b}:

\begin{theorem}
Suppose that $\nu \in C^{0}(\Gamma;\bbR)$ admits an analytic extension on
${\mathcal E}_{\hat \sigma_1, \dots, \hat \sigma_{N}}$ and is absolutely
bounded by $\tilde M(\nu)$.   Let $\mcS^{m,g}_{w}[\nu]$ be the 
sparse grid approximation of the function $\nu$ with Clenshaw-Curtis
abcissas.
If $w > N / \log{2}$ then 
\begin{equation}   
\|\nu - \mcS^{m,g}_{w}[\nu]
\|_{L^{\infty}(\Gamma)} 
\leq 
{\mathcal
  Q}(\sigma,\delta^{*}(\sigma),N, \tilde M(\func))
\eta^{\mu_3(\sigma,\delta^{*}(\sigma),N)}\exp
  \left(-\frac{N \sigma}{2^{1/N}} \eta^{\mu_2(N)} \right)
, \\
\label{erroranalysis:sparsegrid:estimate}
\end{equation}
Furthermore, if $w \leq N / \log{2}$ then the following algebraic
convergence bound holds:
\begin{equation}
  \begin{split}
    \| \nu - \mcS^{m,g}_{w}[\nu] \|_{L^{\infty}(\Gamma)}
    &\leq 
    \frac{C_1(\sigma,\delta^{*}(\sigma),\tilde M(\func))
      \max{\{1,C_{1}(\sigma,\delta^{*}(\sigma),\tilde{M}(\func))
        \}}^N
    }{
|1 - C_1(\sigma,\delta^{*}(\sigma),\tilde M(\func))|
}\eta^{-\mu_1}.
\end{split}
\label{erroranalysis:sparsegrid:estimate2}
\end{equation}
\label{erroranalysis:theorem1}
\end{theorem}
\begin{proof} Theorem 3.10 and 3.11 in \cite{nobile2008a}.
  \end{proof}

Note that in many practical cases not all dimensions of $\Gamma$ are equally
important.  In these cases the dimensionality of the sparse grid can
be significantly reduced by  exploiting the size of each Bernstein ellipse.
By adapting the restriction function $g$ to the $\sigma_n$ along each dimension,
a anisotropic sparse grid is produced with higher convergence rates with
respect to the number of knots $\eta$ \cite{nobile2008b}.

\section{Numerical Results}
\label{numericalresults}

In this section we test the analyticity result by studying the statistical 
behaviour of a quantity of interest of the potential field of the Trypsin protein 
in a solvent. The Trypsin protein molecule (PDB:1ppe \cite{Berman2000}) consists of $n = 1,852$ atoms.
We use the  Adaptive Poisson Boltzmann Solver (APBS) \cite{Baker2001}  for the computation
of potential field on a rectangular domain of size $70 \times 70 \times 70$ \AA ngstroms and
set $\partial \Omega$ with Dirichlet boundary conditions equal to zero. The temperature
of the solvent is set to 310 Kelvin (human body temperature).

Let $\mcC$ be the collection of point charges in the Trypsin molecule.
The Domain $\Omega$ is split into an internal component $\mcA_{I}$, representing
the domain of the molecule that is constructed from the APBS code  and  $\mcA_{E}$, 
which corresponds to the solvent. The point charges $\bx_1, \dots \bx_n$ correspond
to the center of each atom in Trypsin. However, the APBS code replaces these point
charges with functions in $L^{2}(\Omega)$. See Figure \ref{NumericalResults:Fig1} (a)
for a cartoon example.

We now built the stochastic model and shift the point charges in $\mcC$ as follows:
\[
\mcC(\omega) = \{x = x_0 + \sum_{k = 1}^{N} \alpha_{k} \bbe_{k}Y_{k}(\omega) \,|\, x_0 \in \mcC\},
\]
where
\[
\bbe_1 = [1,0,0]^{T},  \bbe_2 = [0,1,0]^{T},  \bbe_3 = [0,0,1]^{T}, \alpha_1 = \alpha_2 = \alpha_3 = 10,
\]
$N \leq 3$, and $Y_1(\omega)$, $Y_2(\omega)$, $Y_3(\omega)$ are uniformly distributed  over the domain $[-\sqrt{3},\sqrt{3}]$ 
and independent of each other. The internal domain $\mcA_{I}(\omega)$ and external domains 
$\mcA_{E}(\omega)$ are then formed. For the internal domain the dielectric $\epsilon_{I}(\omega)$
is assumed to be a constant and $\kappa_{I}(\omega)=0$. For the external domain the dielectric $\epsilon_{E}(\omega)$
and $\kappa_{E}(\omega)$ are assumed to be constants. Note that this is common assumption for
potential field computational packages such as APBS. Furthermore, the point charges $\bx_1(\omega), \dots \bx_n(\omega)$ 
are shifted as follows
\[
\bx_k(x - x_0) \rightarrow \bx_k\left( x - x_0 - \sum_{k = 1}^{N} \alpha_{k} \bbe_{k}Y_{k}(\omega)\right),
\]
Given this stochastic model we computed the expected valued of the spatial mean of the potential field $\eset{Q}$, where
\[
Q := \int_{\Omega} u(Y_1,\dots,Y_N)\,dx
\]
 with respect to stochastic domain shifts.

 Several experiments are now performed:

 \begin{itemize}
     \item $N = 2$ and the dielectric is set to 2, i.e. $\epsilon_I(\omega) = \epsilon_E(\omega) = 2$ (unitless). The     
     Debye-H\"uckel $\kappa$ is set by APBS. The level of the sparse grid is set from $w = 2,3,4,5$. Since we do not have analytic results we assume that $\eset{Q(u)}$ for $w = 5$ is the actual solution. In Figure \ref{NumericalResults:Fig1} (b)
    the convergence graphs of accuracy vs dimensionality of the sparse grid (Work) are plotted for $w = 2,3,4)$. Observe that
    the convergence rate is faster than algebraic. This is consistent with Theorem \ref{erroranalysis:theorem1} and 
    with the result that the solution of the nPBE admits an analytic extension in the complex hyperplane.

    \item The experiment is now repeated with $N = 3$. Observe that the convergence rate is algebraic or faster.

    \item We now extend the experiments for $N = 2$ and $N = 3$ with discontinuous dielectric interface. 
    The internal dielectric is set to 3
    and the external to 2 i.e. $\epsilon_I(\omega) 3$ and $\epsilon_E(\omega) = 2$. From Figure \ref{NumericalResults:Fig1} 
    (c) it is observed that the convergence are almost identical to Figure \ref{NumericalResults:Fig1} (b). Although 
    a discontinuous interface is not covered under the theory developed in this paper, the convergence rates are
    still consistent with the solution of the nPBE admits a complex analytic extension in the complex hyperplane.
 \end{itemize}

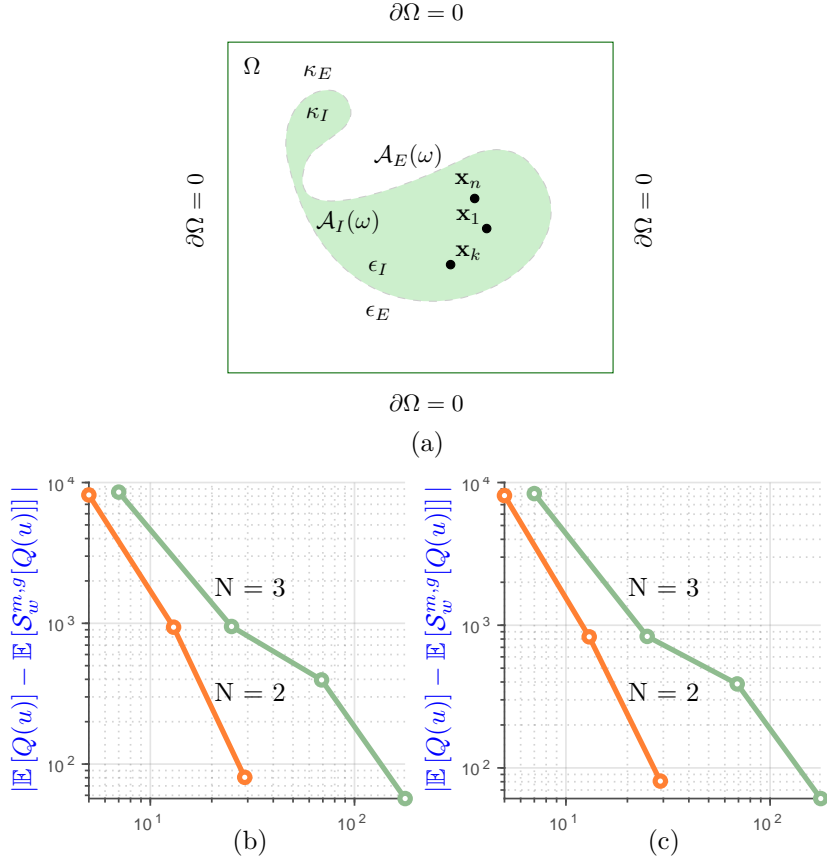
\begin{figure}[htb]
   \centering
   \begin{tikzpicture}
   \begin{scope}[scale = 0.8, every node/.style={scale=0.9}]]   
\draw[dashed,fill=black!30!green, opacity=0.2] (0,0) to
[closed, curve through={(1,1) .. (0.3,0) .. (3,0.1)}] (4,0);
\node at (3,-0.9) {$\bx_1$};
\draw[fill = black] (3.3,-1.1) circle (2pt);
\node at (3,-1.5) {$\bx_{k}$};
\draw[fill = black] (2.7,-1.7) circle (2pt);
\node at (3,-0.3) {$\bx_{n}$};
\draw[fill = black] (3.1,-0.6) circle (2pt);
\node at (1.5,-1.75) {$\epsilon_I$};
\node at (1.5,-2.5) {$\epsilon_E$};
\node at (0.5,0.8) {$\kappa_I$};
\node at (0.5,1.5) {$\kappa_E$};
\node at (1,-1) {${\mcA}_{I}(\omega)$};
\node at (2,0.1) {${\mcA}_{E}(\omega)$};
\node at (2.3,2.5) {$\partial \Omega = 0$};
\node [rotate = 90] at (-1.55,-0.8) {$\partial \Omega = 0$};
\node [rotate = 90] at (5.9,-0.8) {$\partial \Omega = 0$};
\node at (-0.6,1.6) {$\Omega$};
\node at (2.3,-4) {$\partial \Omega = 0$};
\draw[black!60!green]  (-1,-3.5) rectangle (5.4,2);
\end{scope}
\node at (1.85,-3.75) {(a)};
\end{tikzpicture}
   \begin{tikzpicture}[scale = 1, every node/.style={scale=1}]]   
     \node at (0,0) {\includegraphics[scale=0.65, trim = 2.85cm 10cm 2.85cm 10.4cm, clip]{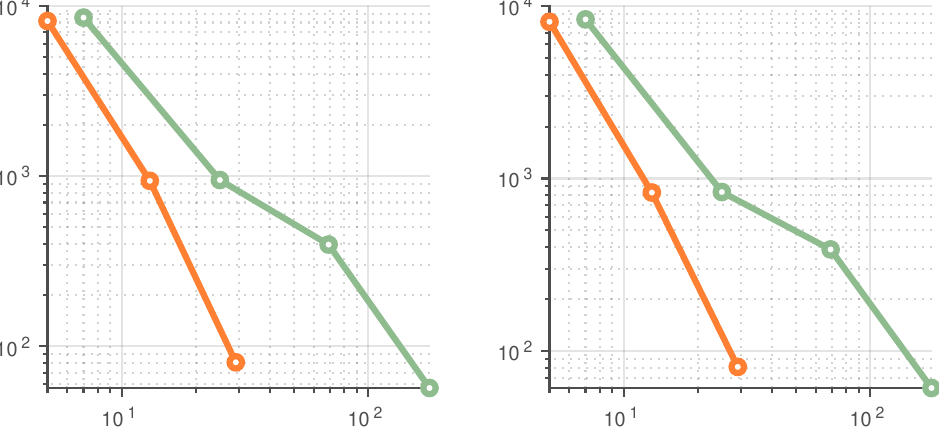}};
     \node at (-2.5,-2.4) {(b)};
     \node at (3,-2.4) {(c)};   
     \node at (-2.5,-.4) {N = 2};
     \node at (-2.5,1) {N = 3};

     \node [rotate = 90] at (-5.5,0.3) {\corb{$|\eset{Q(u)} - \eset{\mcS^{m,g}_{w}[Q(u)]}|$}}; 
     \node [rotate = 90] at (0,0.3) {\corb{$|\eset{Q(u)} - \eset{\mcS^{m,g}_{w}[Q(u)]}|$}}; 
     \node at (3,-.4) {N = 2};
     \node at (3,1) {N = 3};
    \end{tikzpicture}
    \caption{Convergence rates for the Trypsin protein (PDB:1ppe) with stochastic shifts.
    (a) Trypsin protein in a rectangular domain of size $70 \times 70 \times 70$ \AA ngstroms
    with zero Dirichlet boundary conditions. $\bx_1, \dots, \bx_n$ are the centers
    of the atoms with charge dependent on the class of atom. The dielectric coefficient $\epsilon$ is split into internal 
    $\epsilon_{\mcA_{I}}$ and external $\epsilon_{\mcA_{E}}$.
    (b) Accuracy vs dimensionality convergence rates for internal and external dielectric coefficient set
    to 2. Notice that the convergence rate is at least algebraic for $N = 2$ and $N = 3$ stochastic dimensions
    and thus consistent with the existence of an analytic extension into the complex plane for the solution 
    of the nPBE. (c) The experiment is repeated for a discontinuous interface where the internal dielectric is
    set to 3 and the external to 2. Note that the convergence rate is almost identical to the non-discontinuous
    experiment. This indicates that the complex analytic result can be extended to discontinuous interface problems.
    }
    \label{NumericalResults:Fig1}
\end{figure}

 \begin{itemize}
     \item $N = 2$ and the dielectric is set to 2, i.e. $\epsilon_I(\omega) = \epsilon_E(\omega) = 2$ (unitless). The     
     Debye-H\"uckel $\kappa$ is set by APBS. The level of the sparse grid is set from $w = 2,3,4,5$. Since we do not have analytic results we assume that $\eset{Q(u)}$ for $w = 5$ is the actual solution. In Figure \ref{NumericalResults:Fig1} (b)
    the convergence graphs of accuracy vs dimensionality of the sparse grid (Work) are plotted for $w = 2,3,4)$. Observe that
    the convergence rate is faster than algebraic. This is consistent with Theorem \ref{erroranalysis:theorem1} and 
    with the result that the solution of the nPBE admits an analytic extension in the complex hyperplane.

    \item The experiment is now repeated with $N = 3$. Observe that the convergence rate is algebraic or faster.

    \item We now extend the experiments for $N = 2$ and $N = 3$ with discontinuous dielectric interface. 
    The internal dielectric is set to 3
    and the external to 2 i.e. $\epsilon_I(\omega) 3$ and $\epsilon_E(\omega) = 2$. From Figure \ref{NumericalResults:Fig1} 
    (c) it is observed that the convergence are almost identical to Figure \ref{NumericalResults:Fig1} (b). Although 
    a discontinuous interface is not covered under the theory developed in this paper, the convergence rates are
    still consistent with the solution of the nPBE admits a complex analytic extension in the complex hyperplane.
 \end{itemize}

\section{Conclusion.}\label{conclusion}

In this paper, we first showed the existence and uniqueness of the complexified nPBE, and applied the result to show that the solution to the nPBE (and semi-linear elliptic 
PDEs in general) admit a complex analytic extension with respect to the
stochastic variables. Our theoretical predictions were tested by numerically verifying the convergence rate of the sparse grid stochastic collocation approximation with Clenshaw-Curtis abcissas.  

The biggest difference between our model and the real-valued nPBE equation stems from the non-convexity of the nonlinearity on $\mathbb{C}$, which makes it difficult to directly apply variational calculus to our model. See \cite[Theorem 2.14]{Holst1994} that used the variational approach based on convexity to study the real nPBE.


Our future work will focus on the theoretical and numerical studies of nPBE for discontinuous dielectric interface, which more accurately describes the underlying physics. In the Debye-H\"uckel model, a collection of macromolecules such as proteins is located in the region $\Omega_1\subseteq\mathbb{R}^3$, surrounded by the ion-exclusion layer $\Omega_2$, which in turn is surrounded by the solvent of positive and negative charges in $\Omega_3$. Altogether, let $\Omega := \cup_{i=1}^3 \Omega_i$. According to the Debye-H\"uckel model,
\begin{equation}\label{dielectric}
	\epsilon(x) = 
	\begin{cases}
		\epsilon_1>0, &x \in \Omega_1\\
		\epsilon_2>0, &x \in \Omega_2 \cup \Omega_3.
	\end{cases}
\end{equation}
Our analysis assumes that $\epsilon$ is Lipschitz-continuous, and therefore does not cover the case \eqref{dielectric}; the regularity assumption plays a crucial role in obtaining the elliptic regularity results such as \Cref{C_H0} and \Cref{C_H00}. Therefore, it is of interest to develop the functional-analytic framework to study the nPBE with non-trivial interfaces that extends the traditional Debye-H\"uckel model given by \eqref{dielectric}.

\begin{appendices}
\crefalias{section}{appsec}

\section{Estimates for the Constants.}\label{constantest}

In \Cref{Schauder}, the existence and uniqueness of solutions of \Cref{npb} depend in part on the values of the constants \(C_S(s)\), \(C_H\), and \(C_D\) described in \Cref{important_constants}. Thus having explicit estimates for these constants is important in determining the parameter values for which there are solutions. In this section, we demonstrate bounds for these constants. \Cref{h1,h2,h3} are still assumed to hold throughout \Cref{constantest}.

\subsection{Estimates for $C_D$}


	\begin{lemma}
	Let $L$ be as in \Cref{set-up}. Then
	\begin{equation}\label{bounded}
		 C_D \leq 2d^2 \lVert \epsilon \rVert_{W^{1,\infty}} + \lVert \kappa \rVert_{L^\infty}^2.
	\end{equation}
\end{lemma}
\begin{proof}
	Since
	\[
		\lVert \kappa^2 v \rVert_{L^2} \leq \| \kappa \|_{L^\infty}^2 \lVert v \rVert_{L^2} \leq \| \kappa \|_{L^\infty}^2 \lVert v \rVert_{H^2},    
	\]
	it suffices to estimate $\lVert \nabla \cdot (\epsilon \nabla v)\rVert_{L^2}$. By the triangle inequality,
	\[
	\begin{split}
		\lVert \nabla \cdot (\epsilon \nabla v)\rVert_{L^2} &= \lVert \sum_{i,j} \partial_i (\epsilon^{ij}\partial_j v)  \rVert_{L^2} \leq \sum_{i,j} \lVert \partial_i (\epsilon^{ij}\partial_j v) \rVert_{L^2} \leq \sum_{i,j} \lVert \partial_i \epsilon^{ij} \partial_j v \rVert_{L^2} + \lVert \epsilon^{ij} \partial_{ij}v \rVert_{L^2}\\
		&\leq \sum_{i,j} (\lVert \partial_i \epsilon^{ij} \rVert_{L^\infty} + \lVert \epsilon^{ij} \rVert_{L^\infty}) \lVert v \rVert_{H^2} \leq 2d^2 \lVert \epsilon \rVert_{W^{1,\infty}} \lVert v \rVert_{H^2},
		\end{split}
	\]
	and hence \eqref{bounded}.
	
\end{proof}

\subsection{Estimates for $C_S(2)$.}

In this subsection, we are interested in obtaining an upper bound of the operator norm of $H^2(\Omega)\hookrightarrow L^\infty(\Omega)$ where $\Omega \subseteq \mathbb{R}^3$. To obtain this Sobolev inequality constant, a standard trick is to obtain the desired constant for the full domain $\mathbb{R}^d$. Any reasonably regular function defined on $\Omega$ can be extended to $\mathbb{R}^d$ via an extension operator. Composing these two, one obtains a Sobolev inequality on $\Omega$. See \cite[Chapter 5]{evans2010partial} for an exposition of this material. To apply the estimates obtained in \cite{mizuguchi2017estimation}, we lay out the following notation. For $1 \leq p < q \leq \infty$, let $C_{p,q},D_{p,q} >0$ such that for every $v \in W^{1,p}(\Omega)$ and $v_\Omega \coloneqq |\Omega|^{-1}\int_\Omega v$,
\begin{equation}\label{mizuguchi}
	\lVert v \rVert_{L^q(\Omega)} \leq C_{p,q} \lVert v \rVert_{W^{1,p}(\Omega)},\ \lVert v - v_\Omega \rVert_{L^q(\Omega)} \leq D_{p,q} \lVert \nabla v \rVert_{L^p(\Omega)}.
\end{equation}

To estimate $C_{2,p}$ and $C_{p,\infty}$, we cite

\begin{lemma}\cite[Theorem 2.1]{mizuguchi2017estimation}\label{mizuguchi3} For $\Omega \subseteq \mathbb{R}^d$, if $D_{p,q}>0$ is given as \eqref{mizuguchi}, then 
	\[
		C_{p,q} = 2^{1-\frac{1}{p}} \max(|\Omega|^{\frac{1}{q}-\frac{1}{p}},D_{p,q}).
	\] 
\end{lemma}

The estimation for the Sobolev embedding constant, therefore, reduces to computing $D_{p,q}$, which is summarized in the following two lemmas:

\begin{lemma}\cite[Theorem 3.2]{mizuguchi2017estimation}\label{mizuguchi4} Let $p \in (2,6]$ and $v \in H^1(\Omega)$ where we further suppose that $\Omega$ is convex. Then, we have $\lVert v-v_\Omega \rVert_{L^p(\Omega)} \leq D_{2,p} \lVert \nabla v \rVert_{L^2(\Omega)}$ with
	\begin{equation}\label{sobconst1}
		D_{2,p} =\frac{d_\Omega^{1+\frac{3(p+2)}{2p}}\pi^{\frac{3(p+2)}{4p}}}{3|\Omega|} \frac{\Gamma(\frac{3(p-2)}{4p})}{\Gamma(\frac{3(p+2)}{4p})}\sqrt{\frac{\Gamma(\frac{3}{p})}{\Gamma(\frac{3(p-1)}{p})}}\bigg(\frac{4}{\sqrt{\pi}}\bigg)^{\frac{p-2}{2p}}. 
	\end{equation}
	Hence
	\[
		C_{2,p} = 2^{\frac{1}{2}}\max(|\Omega|^{\frac{1}{p}-\frac{1}{2}},D_{2,p}).
	\]
\end{lemma}

\begin{lemma}\cite[Theorem 3.4]{mizuguchi2017estimation}
	\label{mizuguchi5}
	For $p >3$ and $v \in W^{1,p}(\Omega)$, we have $\lVert v-v_\Omega \rVert_{L^\infty} \leq D_{p,\infty} \lVert \nabla v \rVert_{L^p}$ with
	\begin{equation}\label{sobconst3}
		D_{p,\infty} = \frac{d_\Omega^3}{3|\Omega|} \left|\left| |x|^{-2}\right|\right|_{L^{p^\prime}(V)}, 
	\end{equation}
	where $\Omega_x \coloneqq \left\{ x-y: y \in \Omega\right\}$ and $V = \bigcup\limits_{x \in \Omega} \Omega_x$.\footnote{$p^\prime \coloneqq \frac{p}{p-1}$ denotes the H\"older conjugate of $p$.} Hence
	\[
		C_{p,\infty} = 2^{1-\frac{1}{p}}\max(|\Omega|^{-\frac{1}{p}},D_{p,\infty}).
	\]
\end{lemma}

Our proof for the existence of solution depends on the size of the Sobolev inequality constant.

\begin{lemma}\label{sobolevestimate}
	For every $p \in (3,6)$ and $\Omega \subseteq \mathbb{R}^3$ bounded and convex, we have
	\[
		|\Omega|^{-\frac{1}{2}}\leq C_S(2) \leq 2^{\frac{1}{p}} C_{2,p}C_{p,\infty},
	\]
	where for $1 \leq p < q \leq \infty$, denote $C_{p,q}>0$ by a constant such that for every $u \in W^{1,p}(\Omega)$
	\[
		\lVert v \rVert_{L^q(\Omega)} \leq C_{p,q} \lVert v \rVert_{W^{1,p}(\Omega)}.
	\]
	If we further assume that $|\Omega| = Cd_\Omega^3$ for some $C>0$, then there exists $d_0>0$ such that for every $d_\Omega \leq d_0$,
	\[
		|\Omega|^{-\frac{1}{2}} \leq C_S(2) \leq 2^{\frac{3}{2}}|\Omega|^{-\frac{1}{2}}.
	\]
\end{lemma}

\begin{proof}
	Consider the embedding $H^2(\Omega)\hookrightarrow W^{1,p}(\Omega)\hookrightarrow L^\infty(\Omega)$, which is continuous when $p\in (3,6)$. From the first embedding, we obtain
	\[
	\begin{split}
		\lVert v \rVert_{W^{1,p}}^p &= \lVert v \rVert_{L^p}^p + \sum_{i=1}^d \lVert \partial_i v \rVert^p\\
		&\leq C_{2,p}^p \lVert v \rVert_{H^1}^p + C_{2,p}^p \sum_{i=1}^d \lVert \partial_i v \rVert_{H^1}^p \leq C_{2,p}^p \lVert v \rVert_{H^1}^p + C_{2,p}^p \Big(\sum_{i=1}^d \lVert \partial_i v \rVert_{H^1}^2\Big)^{p/2}\\
		&\leq C_{2,p}^p \lVert v \rVert_{H^1}^p + C_{2,p}^p \lVert v \rVert_{H^2}^p \leq 2 C_{2,p}^p \lVert v \rVert_{H^2}^p,
		\end{split}
	\]
	and from the second embedding,
	\[
		\lVert v \rVert_{L^\infty} \leq C_{p,\infty} \lVert v \rVert_{W^{1,p}}.
	\]
	Combining the two, we obtain the upper bound. For the lower bound, consider a family of constant functions defined on $\Omega$. Then
	\[
		\frac{\lVert c \rVert_{L^\infty}}{\lVert c \rVert_{H^2}} = \frac{\lVert c \rVert_{L^\infty}}{\lVert c \rVert_{L^2}} = |\Omega|^{-\frac{1}{2}} \leq C_S(2). 
	\]
	Now we assume $|\Omega| = C d_\Omega^3$ for some $C>0$ and give a sharp bound for $C_S(2)$ for $d_\Omega$ sufficiently small. Note that this hypothesis includes domains such as a ball $B(0,R)\subseteq \mathbb{R}^3$ or a cube $[-R,R]^3$ for $R>0$.
	
	Since $D_{2,p} = C(p)d_\Omega^{\frac{3}{p}-\frac{1}{2}}$ by \Cref{sobconst1} and $|\Omega|^{\frac{1}{p}-\frac{1}{2}} = (C d_\Omega^3)^{\frac{1}{p}-\frac{1}{2}}$, we have
	\begin{equation}\label{sobconst5}
		C_{2,p} = 2^{\frac{1}{2}} |\Omega|^{\frac{1}{p}-\frac{1}{2}},
	\end{equation}
	for all $d_\Omega \leq d_0(p)$ for some $d_0(p)>0$. On the other hand, we may translate the domain and assume $\frac{d_\Omega}{2} = \sup\limits_{x\in \Omega} |x|$. Then, $V \subseteq B(0,d_\Omega)$ and
	\[
		\left|\left| |x|^{-2} \right|\right|_{L^{p^\prime}(V)}^{p^\prime} \leq \left|\left| |x|^{-2} \right|\right|_{L^{p^\prime}(B(0,d_\Omega))}^{p^\prime} = 4\pi \int_0^{d_\Omega} r^{-2p^\prime+2} dr = \frac{4\pi}{-2p^\prime + 3} d_\Omega^{-2p^\prime + 3},
	\]
	and thus $D_{p,\infty} = C^\prime(p)d_\Omega^{-\frac{3}{p}+1}$ by \Cref{sobconst3}, and we have
	\begin{equation}\label{sobconst6}
		C_{p,\infty} = 2^{1-\frac{1}{p}}|\Omega|^{-\frac{1}{p}},
	\end{equation}
	for all $d_\Omega \leq d_0^\prime(p)$ for some $d_0^\prime(p)>0$. Combining \Cref{sobconst5} and \Cref{sobconst6}, we have
	\[
		|\Omega|^{-\frac{1}{2}} \leq C_S(2) \leq 2^{\frac{3}{2}}|\Omega|^{-\frac{1}{2}},
	\]
	for all $d_\Omega$ sufficiently small.
\end{proof}

\subsection{Estimates for $C_H$.}\label{Ch}

To do a numerical simulation, it is of interest to obtain an estimate for the elliptic regularity constant $C_H>0$. In applications, the tensor $\epsilon$ is usually assumed to be a scalar-valued function, in which case, an estimate for $C_H$ can be obtained by the Fourier transform: \[
	\hat{f}(\xi) = \int_{\mathbb{R}^d} f(x)e^{-ix\cdot \xi} d\xi \quad \text{and} \quad f(x) = (2\pi)^{-d}\int_{\mathbb{R}^d} \hat{f}(\xi)e^{i x \cdot \xi} d\xi.
\] For any $s\in \mathbb R$, define
\[
	H^s(\mathbb{R}^d) = \{f \in \mathscr{S}^\prime: \langle \xi \rangle^s \hat{f} \in L^2(\mathbb{R}^d)\},
\]
where $\langle \xi \rangle \coloneqq (1+|\xi|^2)^{\frac{1}{2}}$ and $\mathscr{S}^\prime$ is the space of tempered distributions.

\begin{lemma}\label{C_H0}
	Let \(L\) be as in \Cref{set-up}. Furthermore, suppose $\epsilon^{ij} = \epsilon(x)\delta_{ij}$ where $\delta_{ij}$ is the Kronecker delta function and $\epsilon \in W^{1,\infty}(\Omega)$ such that $\Real(\epsilon(x)) \geq \theta>0$ for all $x \in \Omega$. Then
\begin{equation}\label{ch.estimate}
C_H \leq \frac{\lambda_1^{-1} \langle \lambda_1^{\frac{1}{3}} \rangle^3 }{\theta}\left(1+ \frac{\lVert \kappa^2 \rVert_{L^\infty(\Omega)}+d^{\frac{1}{2}}\max\limits_{1 \leq i \leq d} \lVert \partial_i \epsilon \rVert_{L^\infty(\Omega)}\lambda_1^{\frac{1}{2}}}{\theta \lambda_1 - \mu}\right).
\end{equation}
\end{lemma}
\begin{proof}
	Given $F \in L^2(\Omega)$ and a unique weak solution $u \in H^1_0(\Omega)$ of the Laplace equation $-\Delta u = F$ in $\Omega$, we find \(C_1>0\) such that $\lVert u \rVert_{H^2(\Omega)} \leq C_1 \lVert F \rVert_{L^2(\Omega)}$. We use this energy estimate to handle the more complicated case.
	
	By the density argument, it suffices to assume $F \in C^\infty_c(\Omega)$. By an integration-by-parts argument, it can be shown that $u \in C^2_c(\Omega)$. Hence, we extend $u$ to a function in $C^2_c(\mathbb{R}^d)$, which we continue to call $u$, by defining $u(x) = 0$ for $x \in \mathbb{R}^d \setminus \mathrm{supp}(u).$ Then,
	\begin{equation*}
	\begin{aligned}
		\lVert u \rVert_{H^2(\Omega)}^2 &\leq \lVert u \rVert_{H^2(\mathbb{R}^d)}^2= \int_{\mathbb{R}^d} \langle \xi \rangle^4 |\hat{u}(\xi)|^2 d\xi  
		\\&= \int_{|\xi| \geq c} \langle \xi \rangle^4 |\hat{u}(\xi)|^2 d\xi + \int_{|\xi|<c} \langle \xi \rangle^4 |\hat{u}(\xi)|^2 d\xi =: I + II
	\end{aligned}
	\end{equation*}
	for some \(c>0\) to be fixed later. For the high frequencies,
    \[
	\begin{split}
		I &= \int_{|\xi| \geq c} \langle \xi \rangle^4 |\hat{u}(\xi)|^2 d\xi = \int_{|\xi|\geq c} \frac{\langle \xi \rangle^4}{|\xi|^4} |\widehat{\Delta u}|^2 d\xi= \int_{|\xi|\geq c} \frac{\langle \xi \rangle^4}{|\xi|^4} |\hat{F}(\xi)|^2 d\xi \leq \frac{\langle c \rangle^4}{|c|^4} \lVert F \rVert_{L^2(\Omega)}^2.
	\end{split}
	\]
	Combining the Poincar\'e inequality and the weak form of the Laplace equation, we have
	\[
		\lVert u \rVert_{L^2(\Omega)}^2 \leq \lambda_1^{-1} \lVert \nabla u \rVert_{L^2(\Omega)}^2 \leq \lambda_1^{-1}\lVert F \rVert_{L^2(\Omega)} \lVert u \rVert_{L^2(\Omega)},
	\]
	and therefore, for the low frequencies,
	\[
		II \leq \langle c \rangle^4 \lVert u \rVert_{L^2(\Omega)}^2  \leq \langle c \rangle^4 \lambda_1^{-2} \lVert F \rVert_{L^2(\Omega)}^2.
	\]
	Combining $I$ and $II$,
	\[
		\lVert u \rVert_{H^2(\Omega)} \leq \langle c \rangle^2 (|c|^{-4}+ \lambda_1^{-2})^{\frac{1}{2}} \lVert F \rVert_{L^2(\Omega)}.
	\]
	Noting that $c \mapsto \langle c \rangle^2 (|c|^{-4}+ \lambda_1^{-2})^{\frac{1}{2}}$ has a global minimum at $c = \lambda_1^{\frac{1}{3}}$, we fix that value of $c$ to obtain
	\begin{equation}\label{C_H2}
		\lVert u \rVert_{H^2(\Omega)} \leq C_1 \lVert F \rVert_{L^2(\Omega)} \quad \text{where} \quad C_1 \coloneqq \lambda_1^{-1} \langle \lambda_1^{\frac{1}{3}} \rangle^3 .
	\end{equation}
	Now we assume $u \in H^1_0(\Omega)$ is the unique weak solution of 
	\begin{equation}\label{C_H}
		-\nabla \cdot (\epsilon\nabla u) + \kappa^2 u = f \quad \text{in } \Omega.
	\end{equation}
	Setting $F \coloneqq f- \kappa^2 u \in L^2(\Omega)$, the product rule applied to \Cref{C_H} yields 
	\[
		-\Delta u = \epsilon(x)^{-1}(F + \nabla \epsilon \cdot \nabla u).    
	\]
	Noting that $|\epsilon(x)| \geq |\Real (\epsilon(x))| \geq \Real (\epsilon(x)) \geq \theta$, an immediate application of \Cref{C_H2} yields
	\begin{equation}\label{C_H3}
		\lVert u \rVert_{H^2(\Omega)} \leq \frac{C_1}{\theta} (\lVert F \rVert_{L^2(\Omega)} + \lVert \nabla \epsilon \cdot \nabla u\rVert_{L^2(\Omega)}).
	\end{equation}
	Taking the real part of the weak form of \Cref{C_H}, we have
	\[
		\int_{\Omega} \Real(\epsilon(x)) |\nabla u|^2 + \int_{\Omega} \Real(\kappa^2) |u|^2 = \Real \int_{\Omega} f\overline{u}.
	\]
	Recalling that $\Real(\kappa^2(x)) \geq - \mu$ for all $x \in \Omega$ and the uniform ellipticity,
	\begin{equation}\label{poincare3}
		\theta \int_{\Omega} |\nabla u|^2 + \int_{\Omega} \Real(\kappa^2) |u|^2 \leq \lVert f \rVert_{L^2(\Omega)} \lVert u \rVert_{L^2(\Omega)}.
	\end{equation}
	The Poincar\'e inequality yields $ (\theta \lambda_1 - \mu)\lVert u \rVert_{L^2(\Omega)}^2 \leq \theta \int_{\Omega} |\nabla u|^2 + \int_{\Omega} \Real(\kappa^2) |u|^2$, which gives
	\[
		\lVert u \rVert_{L^2(\Omega)} \leq \frac{\lVert f \rVert_{L^2(\Omega)}}{\theta \lambda_1 -\mu}.
	\]
	Another application of the Poincar\'e inequality to \Cref{poincare3} yields
 \begin{equation*}
     (\theta - \mu \lambda_1^{-1}) \int_\Omega |\nabla u|^2 \leq \theta \int_{\Omega} |\nabla u|^2 + \int_{\Omega} \Real(\kappa^2) |u|^2,
 \end{equation*}
which gives
	\begin{equation}\label{l22}
		\lVert \nabla u \rVert_{L^2(\Omega)} \leq \lambda_1^{\frac{1}{2}}\frac{\lVert f \rVert_{L^2(\Omega)}}{\theta \lambda_1 - \mu}.
	\end{equation}
	Hence,
	\begin{equation}\label{C_H4}
		\lVert F \rVert_{L^2(\Omega)} \leq \Big(1+\frac{\lVert \kappa^2 \rVert_{L^\infty(\Omega)}}{\theta \lambda_1 - \mu}\Big) \lVert f \rVert_{L^2(\Omega)}.
	\end{equation}
	On the other hand, the Cauchy-Schwarz inequality yields
	\begin{equation}\label{C_H5}
		\lVert \nabla \epsilon \cdot \nabla u\rVert_{L^2(\Omega)} \leq \sqrt{d} \max_{1 \leq i \leq d} \lVert \partial_i \epsilon\rVert_{L^\infty(\Omega)} \lVert \nabla u \rVert_{L^2(\Omega)} \leq \sqrt{d} \max_{1 \leq i \leq d} \lVert \partial_i \epsilon\rVert_{L^\infty(\Omega)}\lambda_1^{\frac{1}{2}}\frac{\lVert f \rVert_{L^2(\Omega)}}{\theta \lambda_1 - \mu}.
	\end{equation}
	By \Cref{C_H2,C_H3,C_H4,C_H5},
	\[
		\lVert u \rVert_{H^2(\Omega)} \leq \frac{\lambda_1^{-1} \langle \lambda_1^{\frac{1}{3}} \rangle^2 (1+ \lambda_1^{\frac{2}{3}})^{\frac{1}{2}}}{\theta}\Big(1+ \frac{\lVert \kappa^2 \rVert_{L^\infty(\Omega)}+d^{\frac{1}{2}}\max\limits_{1 \leq i \leq d} \lVert \partial_i \epsilon \rVert_{L^\infty(\Omega)}\lambda_1^{\frac{1}{2}}}{\theta \lambda_1 - \mu}\Big) \lVert f \rVert_{L^2(\Omega)}.
	\]
\end{proof}

In general when $\epsilon$ is a tensor, a direct application of Fourier transform seems infeasible. Instead, we closely follow the argument of \cite[Section 6.3, Theorem 4]{evans2010partial} to obtain an estimate on $C_H$. An emphasis here is that we keep track of the implicit constants.

\begin{lemma}\label{C_H00}
	Let \(L\) be as given in \Cref{set-up}. Then
\begin{equation}\label{ch.estimate2}
C_H \leq N(\Omega)\Big(\Big(\frac{1+\lambda_1}{\theta \lambda_1 - \mu} \Big)^2 + d C_0\Big)^{\frac{1}{2}},     
\end{equation}
where $N(\Omega) \in \mathbb{N}$ and $C_0$ is defined in \Cref{C_0}.
\end{lemma}
\begin{remark}\label{ch_diameter}
Since $\lambda_1 \geq \frac{\pi^2}{d_\Omega^2}$ for convex $\Omega$, the RHS of \Cref{ch.estimate,ch.estimate2} converges to $\frac{C(\Omega)}{\theta}$ as $d_\Omega \rightarrow 0$.
\end{remark}
\begin{remark}
Since the estimate of \Cref{ch.estimate2} depends on the number of finitely many open balls covering $\Omega$, the geometry of $\partial \Omega$ plays a big role in the computation of $N(\Omega)$. This will be pursued in future research. 
\end{remark}
\begin{proof}
	Let $V \Subset W \Subset \Omega$. Let $\zeta \in C^\infty_c(\Omega)$ such that $0\leq \zeta \leq 1$ and $\zeta = 1$ on $V$, and $\mathrm{supp}(\zeta)\subseteq W$. Set $F = f- \kappa^2 u \in L^2(\Omega)$. Consider the weak form of 
    \begin{equation}
		-\nabla \cdot (\epsilon\nabla u) + \kappa^2 u = f \quad \text{in } \Omega
	\end{equation}
	applied to the test function $\phi = -D_k^{-h} \zeta^2 D_k^h u$, where 
	\[D_k^h u(x)= \frac{u(x+he_k) - u(x)}{h}\] and $\{e_k\}_{k=1}^d$ forms the standard basis of $\mathbb{R}^d$. Using integration by parts and the product rule of discrete derivatives,
	\begin{equation*}
	\begin{aligned}
		\int_\Omega \epsilon^{ij} \partial_j u \overline{\partial_i \phi} &= \int_\Omega D_k^h (\epsilon^{ij}\partial_j u) \overline{\partial_i (\zeta^2 D_k^h u)} \\
		&= \int_{\Omega} (\epsilon^{ij,h} D_k^h \partial_j u + D_k^h \epsilon^{ij} \partial_j u) (\overline{2\zeta \partial_i \zeta D_k^h u} + \overline{\zeta^2 D_k^h \partial_i u})\\
		&= \int_\Omega \zeta^2 \epsilon^{ij,h} D_k^h \partial_j u \overline{D_k^h \partial_i u} + R,
	\end{aligned}
	\end{equation*}
	where $\epsilon^{ij,h}(x) \coloneqq \epsilon^{ij}(x+he_k)$. By uniform ellipticity,
	\[
		\Real\int_\Omega \zeta^2 \epsilon^{ij,h} D_k^h \partial_j u \overline{D_k^h \partial_i u} \geq \theta \int_\Omega \zeta^2 |D_k^h \nabla u|^2.
	\]
	The other three products are estimated above by the Cauchy-Schwarz inequality:
	\begin{equation}\label{R}
	\begin{split}
		R &\leq \left|\int_\Omega 2 \zeta \partial_i \zeta \epsilon^{ij,h} D_k^h \partial_j u \overline{D_k^h u}\right| + \left| \int_\Omega 2\zeta \partial_i \zeta D_k^h \epsilon^{ij} \partial_j u \overline{D_k^h u}\right| + \left|\int_\Omega D_k^h \epsilon^{ij} \zeta^2 \partial_j u \overline{D_k^h \partial_i u} \right| \\
		&\leq 2 \lVert \nabla \zeta \rVert_{L^\infty(\Omega)} \lVert \epsilon \rVert_{W^{1,\infty}(\Omega)}\Big( \int_\Omega \zeta |D_k^h \nabla u||D_k^h u| + \int_\Omega \zeta |\nabla u||D_k^h u|\Big)+ \lVert \epsilon \rVert_{W^{1,\infty}(\Omega)} \int_\Omega \zeta |\partial_j u| |D_k^h \partial_i u|.
	\end{split}
	\end{equation}
	Recalling the following variant of Cauchy-Schwarz inequality
	\[
		ab \leq \frac{a^2}{2\delta} + \frac{\delta b^2}{2},
	\]
	for $a,b \geq 0$ and $\delta>0$ and the following control of discrete derivatives with respect to the continuous derivatives for sufficiently small $|h|>0$,
	\begin{equation}\label{discretederiv}
		\lVert D_k^h \phi\rVert_{L^2(V)} \leq \lVert \partial_k \phi \rVert_{L^2(\Omega)},\:\forall \phi \in H^1(\Omega),\: V \Subset \Omega,
	\end{equation}
	\Cref{R} is bounded above by
	\[
		\leq C_1\delta \int_\Omega \zeta^2 |D_k^h \nabla u|^2 + C_2\int_{\Omega} |\nabla u|^2
	\]
	where
	\[
	\begin{split}
		C_1 &\coloneqq \lVert \epsilon \rVert_{W^{1,\infty}(\Omega)}\Big(\lVert \nabla \zeta \rVert_{L^\infty(\Omega)} + \frac{1}{2}\Big) \quad \text{and} \quad C_2 \coloneqq \lVert \epsilon \rVert_{W^{1,\infty}(\Omega)}\Big(2\lVert \nabla \zeta \rVert_{L^\infty(\Omega)} +\frac{1+2\lVert \nabla \zeta \rVert_{L^\infty(\Omega)}}{2\delta}\Big).
	\end{split}
	\]
	Choosing $\delta = \frac{\theta}{2C_1}$, we use the triangle inequality to obtain
	\begin{equation}\label{lowerbound2}
		\Real\int_\Omega (\epsilon \nabla u) \cdot \overline{\nabla u} \geq \frac{\theta}{2}\int_\Omega \zeta^2 |D_k^h \nabla u|^2 - C_2 \int_\Omega |\nabla u|^2. 
	\end{equation}
	On the other hand, we estimate the right-hand side of the weak form:
	\[
		\left|\int_\Omega F \overline{\phi}\right| \leq \frac{1}{2\delta} \int_\Omega |F|^2 + \frac{\delta}{2}\int_\Omega |\phi|^2,
	\]
	where the first term is estimated above as in \Cref{C_H4}.
	\begin{equation*}
	\begin{aligned}
		\int_\Omega |\phi|^2 &\leq \int_\Omega |\partial_k (\zeta^2 D_k^h u)|^2 \\
		&\leq 2 \int_\Omega |2\zeta \partial_k \zeta D_k^h u|^2 + 2 \int_\Omega \zeta^2 |D_k^h \partial_k u|^2\\
		&\leq 8 \lVert \nabla \zeta \rVert_{L^\infty(\Omega)}^2 \int_\Omega |\nabla u|^2 + 2 \int_\Omega \zeta^2 |D_k^h \nabla u|^2,
	\end{aligned}
	\end{equation*}
	where the last inequality is by \Cref{discretederiv}. Let $\delta = \frac{\theta}{4}$. Then,
	
	\begin{equation}\label{lowerbound3}
		\left| \int_\Omega F \overline{\phi}\right| \leq \frac{2}{\theta} \Big(1+\frac{\lVert \kappa^2 \rVert_{L^\infty(\Omega)}}{\theta \lambda_1  - \mu} \Big)^2 \int_\Omega |f|^2 + \theta \lVert \nabla \zeta \rVert_{L^\infty(\Omega)}^2 \int_\Omega |\nabla u|^2 + \frac{\theta}{4} \int_\Omega \zeta^2 |D_k^h \nabla u|^2.
	\end{equation}
	Combining \Cref{lowerbound3} and \Cref{lowerbound2},
	\[
		\frac{\theta}{4} \int_V |D_k^h \nabla u|^2 \leq \frac{\theta}{4} \int_\Omega \zeta^2 |D_k^h \nabla u|^2 \leq \frac{2}{\theta}\Big(1+ \frac{\lVert \kappa^2 \rVert_{L^\infty(\Omega)}}{\theta \lambda_1 - \mu} \Big)^2 \int_\Omega |f|^2 + (C_2 + \theta \lVert \nabla \zeta \rVert_{L^\infty(\Omega)}^2) \int_\Omega |\nabla u|^2,
	\]
	and by \Cref{l22},
	\begin{equation}
	\begin{split}
	\label{C_0}
		\int_V |D_k^h \nabla u|^2 &\leq C_0 \int_\Omega |f|^2, \\
		C_0 &= \frac{4}{\theta}\bigg(\frac{2}{\theta}\Big(1+ \frac{\lVert \kappa^2 \rVert_{L^\infty(\Omega)}}{\theta \lambda_1 - \mu} \Big)^2+ \frac{\lambda_1 (C_2 + \theta \lVert \nabla \zeta \rVert_{L^\infty(\Omega)}^2)}{(\theta \lambda_1 - \mu)^2} \bigg).
	\end{split}
	\end{equation}
	By \cite[Section 5.8.2, Theorem 3]{evans2010partial}, this shows $\partial_k \nabla u \in L^2(V,\mathbb{C}^d)$ for all $1 \leq k \leq d$ with the same bound on the $L^2$-norm. Hence,
	\begin{equation}\label{h2homogeneous}
		\sum_{1\leq i,j \leq d} \lVert \partial_{ij} u \rVert_{L^2(V)}^2 \leq d C_0 \int_\Omega |f|^2.
	\end{equation}
	By the Lax-Milgram Theorem, we have
 \begin{equation}\label{LM3}
		\lVert u \rVert_{H^1(\Omega)} \leq \frac{1+\lambda_1}{\theta \lambda_1 - \mu} \lVert f \rVert_{L^2(\Omega)}.
	\end{equation}
	Combining \Cref{h2homogeneous} with \Cref{LM3}, we obtain
	\begin{equation}\label{h2inhomogeneous}
		\lVert u \rVert_{H^2(V)} \leq \Big(\Big(\frac{1+\lambda_1}{\theta \lambda_1 - \mu} \Big)^2 + d C_0\Big)^{\frac{1}{2}}\lVert f \rVert_{L^2(\Omega)}.
	\end{equation}
	Since $\Omega$ is bounded, $\{x \in \Omega: \inf_{y \in \partial\Omega}|x-y| \geq \delta\}$ can be covered by finitely many open sets for every $\delta>0$. Given any point $y \in \partial \Omega$, there exists a diffeomorphism that takes a small neighborhood of $y$ (in $\overline{\Omega}$) into a neighborhood in the half-plane $\mathbb{R}^d_+ \coloneqq \mathbb{R}^{d-1}\times [0,\infty)$ where $y$ is identified with $0 \in \mathbb{R}^d_+$. Via this diffeomorphism, one can show that the $H^2$-norm of $u$ in the neighborhood of $y$ obeys an esmate similar to \Cref{h2inhomogeneous}. Hence, there exists $N = N(\Omega) \in \mathbb{N}$ such that
	\[
		\lVert u \rVert_{H^2(\Omega)} \leq N(\Omega)\Big(\Big(\frac{1+\lambda_1}{\theta \lambda_1 - \mu} \Big)^2 + d C_0\Big)^{\frac{1}{2}} \lVert f \rVert_{L^2(\Omega)}.
	\]
\end{proof}

\section{Failure of Uniqueness}\label{appendixB}
In this appendix, we demonstrate that the assumption made in proving the existence of unique solutions were reasonable. In particular, if \cref{h2,h3} or the smallness assumptions on \((f,g)\in L^2(\Omega) \times H^{3/2}(\Omega)\) are violated, then there can be multiple small solutions of the nPBE.

Let us first examine what can happen if we allow \cref{h2,h3} to be violated. For the case where \(\epsilon\) and \(\kappa^2\) are scalar-valued and \(f\) and \(g\) set to zero functions, \cref{npb} simplifies to 

\begin{equation}\label{scalar_NPBE}
	\begin{aligned}
		-\Delta u + (\eta - \lambda_1) \sinh(u) &= 0 \quad \text{ in } \Omega, \\
		u & = 0 \quad \text{ on } \partial \Omega,
	\end{aligned}
\end{equation}
where \(\eta = \lambda_1 + \kappa^2/\epsilon\in \mathbb{R}\). The function \(u\equiv 0\) is always a trivial solution for the above PDE. For \Cref{scalar_NPBE}, \Cref{h2,h3} are satisfied if and only if \(\eta\) is greater than zero. The parameter \(\eta\) is the smallest eigenvalue of the \(-\Delta  + (\eta-\lambda_1)\), so this linear operator is non-invertible when \(\eta = 0\). The following result proved by Crandall and Rabinowitz in \cite{crandall1971bifurcation} can be used to show that the zero solution undergoes a bifurcation at \(\eta=0\):

\begin{theorem}\label{CR_theorem}
	Let \(X\) and \(Y\) be Banach spaces and assume 
	\begin{enumerate}[label= (\roman*)]
		\item \(F \in C^2(X\times \mathbb R, Y)\),
		\item \(F(0,\eta) = 0\) for all \(\eta \in \mathbb R\),
		\item \(\dim{N(D_xF(0,0))} = \mathrm{codim}{R(D_x F(0,0) )} = 1\), and 
		\item \(D^2_{x\eta} F(0,0) \hat{v}_0 \notin R(D_xF(0,0))\) where \(\hat{v}_0\neq 0\)  is in  \(N(D_xF(0,0))\). \label{item_4}
	\end{enumerate}
	Then there is a nontrivial continuously differentiable curve
	\begin{equation}
		\{(x(s),\eta(s)) \mid s\in (-\delta, \delta) \}
	\end{equation}
	such that \((x(0), \eta(0)) = (0,0)\), \(x'(0) = \hat{v}_0\) and 
	\begin{equation}
		F(x(s), \eta(s)) = 0 \quad \text{ for } \quad s\in(-\delta, \delta).
	\end{equation}
\end{theorem}

Here \(D_x\) and \(D_\eta\) represent the Frechet derivatives of \(F\) with respect to the \(X\) and \(\mathbb R\) components, respectively. Note that \(D_\eta F(x,\eta) \in \mathcal L (\mathbb R, Y)\), the set of linear operators from \(\mathbb R\) into \(Y\). An element \(A \in \mathcal L (\mathbb R,Y)\) can be uniquely associated with an element \(y\in Y\) by setting \(y\in A(1)\). Thus \(D_{x\eta}^2F(x,\eta)\) can be associated with an element of \(\mathcal L(X,Y)\), which is how the map \(D^2_{x\eta}F(0,0)\) is being viewed in \Cref{item_4}. Applying \Cref{CR_theorem} gives the existence of non-unique small solutions.

\begin{theorem}\label{nonunique_solutions}
	Let \(c > 0\). Then there is \(\eta^* < 0\) such that for any \(\eta \in [\eta^*, 0)\) there is a non-trivial solution \(u\) of \Cref{scalar_NPBE} such that \(\|u\|_{H^2} < c\).
\end{theorem}

\begin{proof}
	Define \(F:H^2(\Omega)\cap H^1_0(\Omega) \times \mathbb R \to L^2(\Omega)\) as 
	\begin{equation}
		F(u,\eta) = -\Delta u + (\eta - \lambda_1) \sinh(u).
	\end{equation}
	Assumptions \((i)\) and \((ii)\) are clearly satisfied. Note that \(D_u F(0,0) = -\Delta - \lambda_1\). Thus assumption \((iii)\) follows from the Fredholm properties of \(-\Delta\) and from the fact that \(\lambda_1\) is the principal eigenvalue of \(-\Delta\). Let \(\hat{v}_0\) be be the eigenfunction corresponding to \(\lambda_1\). Then assumption \((iv)\) states that \(D_{u\eta}^2 F(0,0) \hat{v}_0 = \hat{v}_0 \notin R(D_u F(0,0))\). This should hold since if \(\hat{v}_0 \in R(D_u F(0,0))\), then there is a generalized eigenfunction for \(\lambda_1\) which contradicts the simplicity of \(\lambda_1\). Hence, we can apply \Cref{CR_theorem} to \(F(u,\eta) = 0\) to get a nontrivial continuously differentiable curve
	\begin{equation}\label{curve}
		\{(u(s), \eta(s)) \mid s\in(-\delta,\delta)\}.
	\end{equation}
	with \((u(0), \eta(0)) = (0,0)\).
	
	One can also compute derivatives of \(\eta(s)\) (see \cite[\S 1.6]{kielhofer2011bifurcation} for details) to get that \(\eta'(0)  = 0\) and \(\eta''(0) < 0\). Thus it is possible to choose \(\delta\) small enough so that \(\eta(s)\) is strictly decreasing on \(s\in[0,\delta)\). Given \(c>0\), we can find \(s^*>0\) small enough to get \(\|u(s)\|_{H^2} \leq c\) for all \(s\in (0, s^*]\) by the continuity of the curve. Since \(s\mapsto \eta(s)\) is strictly decreasing on \((0,s^*]\), we can invert this mapping to get \(\eta\mapsto s(\eta)\) for \(\eta\in [\eta^*,0)\) where \(\eta^* = \eta(s^*).\) Therefore, for each \(\eta\in [\eta^*,0)\) we have a solution of \Cref{scalar_NPBE}, given by \(u = u(s(\eta))\) such that \(\|u\|_{H^2}<c\).
\end{proof}

\begin{remark}
	There are two non-trivial solutions to \Cref{scalar_NPBE} for \(\eta< 0\) sufficiently close to zero: one for \(s>0\) and one for \(s<0\). In fact, from the oddness of \(\sinh(u)\), if \(u\) is a solution of \Cref{scalar_NPBE} then so is \(-u\).
\end{remark}

Comparing \Cref{nonunique_solutions} with \Cref{Schauder} shows that uniqueness cannot be guaranteed without \Cref{h2,h3}.  

We may also lose uniqueness of solutions if $(f,g)$ become too large. We construct an example of nPBE that admits multiple solutions. By construction, this family of nPBEs fails to satisfy the invertibility condition given in \Cref{h3} and/or the smallness assumption on $(f,g) \in L^2(\Omega)\times H^{\frac{3}{2}}(\partial \Omega)$. In particular, this example is consistent with the well-known uniqueness result of \cite{kwong1989uniqueness}.

We wish to obtain a radial solution $u(x) = u(|x|) = y(r)$, where $r=|x|\geq 0$, to \Cref{npb} where $\epsilon=1$ for simplicity and $\kappa = i\tilde{\kappa} \in i \mathbb{R}$ on domain $\Omega = B(0,R)\subseteq \mathbb{R}^d$ for $R>0, d\geq 1$ and $f(x) = \lambda \in \mathbb{R}$, $g(x) = \sinh^{-1}(\frac{\lambda}{\tilde{\kappa}^2})$. In the polar coordinate, our example reduces to an ODE
\begin{equation}\label{nonlinearBessel}
\begin{split}
		r y^{\prime\prime}+(d-1)y^\prime + \tilde{\kappa}^2 r\sinh y&=r\lambda,\: r \in (0,R)\\
		y(R) &= \sinh^{-1}\Big(\frac{\lambda}{\tilde{\kappa}^2}\Big).
\end{split}
\end{equation}
where it is clear that the constant function $r\mapsto \sinh^{-1}\Big(\frac{\lambda}{\tilde{\kappa}^2}\Big)$ is a trivial solution. Since \Cref{nonlinearBessel} is symmetric under $r\mapsto -r$, we may consider $\lambda \geq 0$. It is also clear that $(f,g) \in L^2(\Omega)\times H^{\frac{3}{2}}(\partial\Omega)$ can be taken as large as possible (in norm) by taking $\lambda$ arbitrarily large. Furthermore, we note that \Cref{h3} is violated when $R \gg 1$ depending on $\tilde{\kappa}$. To elaborate, fix $\tilde{\kappa}>0$. If \Cref{h3} holds, then $\tilde{\kappa}^2 \leq \mu < \lambda_1 = \frac{C_B}{R^2}$. Hence if $R >\frac{C_B}{\tilde{\kappa}}$, then \Cref{h3} cannot hold.

\begin{proposition}\label{nonunique}
	Let $d\geq 1, \tilde{\kappa}>0, \lambda \geq 0$. Then, there exists a non-trivial solution to \Cref{nonlinearBessel} with $R>\frac{C_B}{\tilde{\kappa}}$.
\end{proposition}
	
Reducing \Cref{nonlinearBessel} into a first-order ODE by introducing $w = y^\prime$, we obtain \begin{equation}\label{phaseportrait}
		\begin{pmatrix}
			y\\
			w
		\end{pmatrix}^\prime = F(r,y,w)\coloneqq
		\begin{pmatrix}
			w\\
			-\tilde{\kappa}^2\sinh y - (d-1)\frac{w}{r} + \lambda
		\end{pmatrix}.
	\end{equation}
	
	For $d=1$, \Cref{phaseportrait} admits an autonomous Hamiltonian vector field where the Hamiltonian is given by 
	\[
		H(y,w) = \frac{w^2}{2} + \tilde{\kappa}^2(\cosh y -1) - \lambda y.    
	\]
	Since the level sets of $H$ are a collection of closed one-dimensional curves, all solutions are global and periodic. The inner curves have lower values of $H$ than the outer curves. Indeed, the global minimum of $H$ occurs at $P=(\sinh^{-1}\Big(\frac{\lambda}{\tilde{\kappa}^2}\Big),0)$ where $H(P)\leq 0$ with the equality if and only if $\lambda = 0$. Hence for each initial datum $\begin{pmatrix}
		c\\
		0
	\end{pmatrix}$, there exists a unique solution $y$ to \Cref{nonlinearBessel} where $y(R) = 0$ for infinitely many $R>0$. We include a phase portrait where the solutions lie on the curves of constant Hamiltonian.
	
	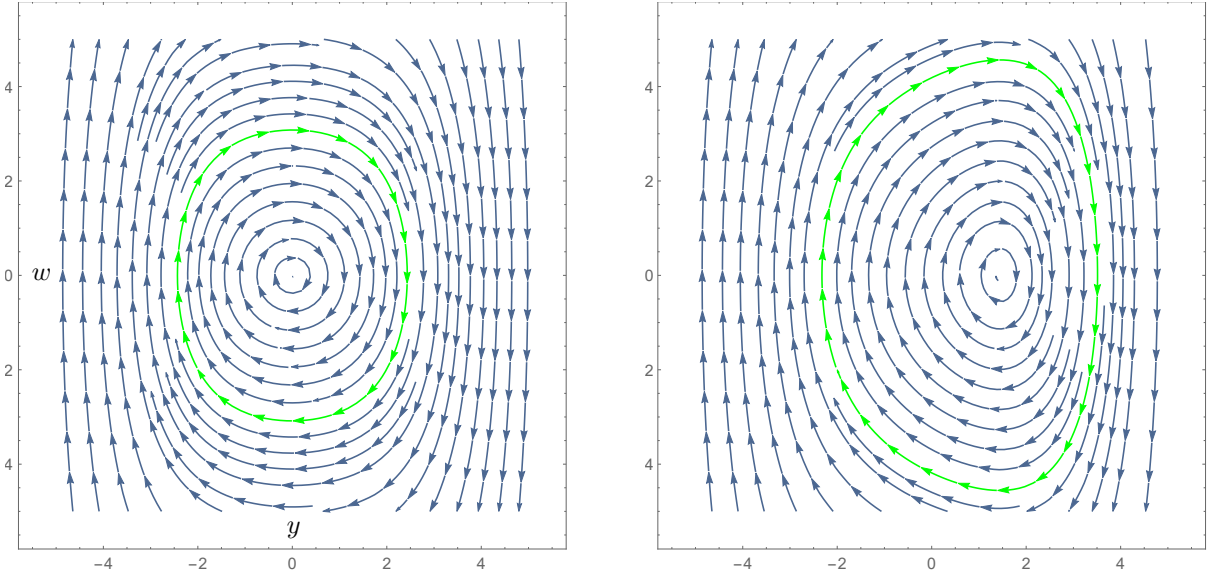
\begin{figure}[H]
	\begin{tikzpicture}
	
		\node at (0,0) {\includegraphics[scale=0.3, trim = 0.5cm 0.15cm 0.05cm 0.05cm, clip]{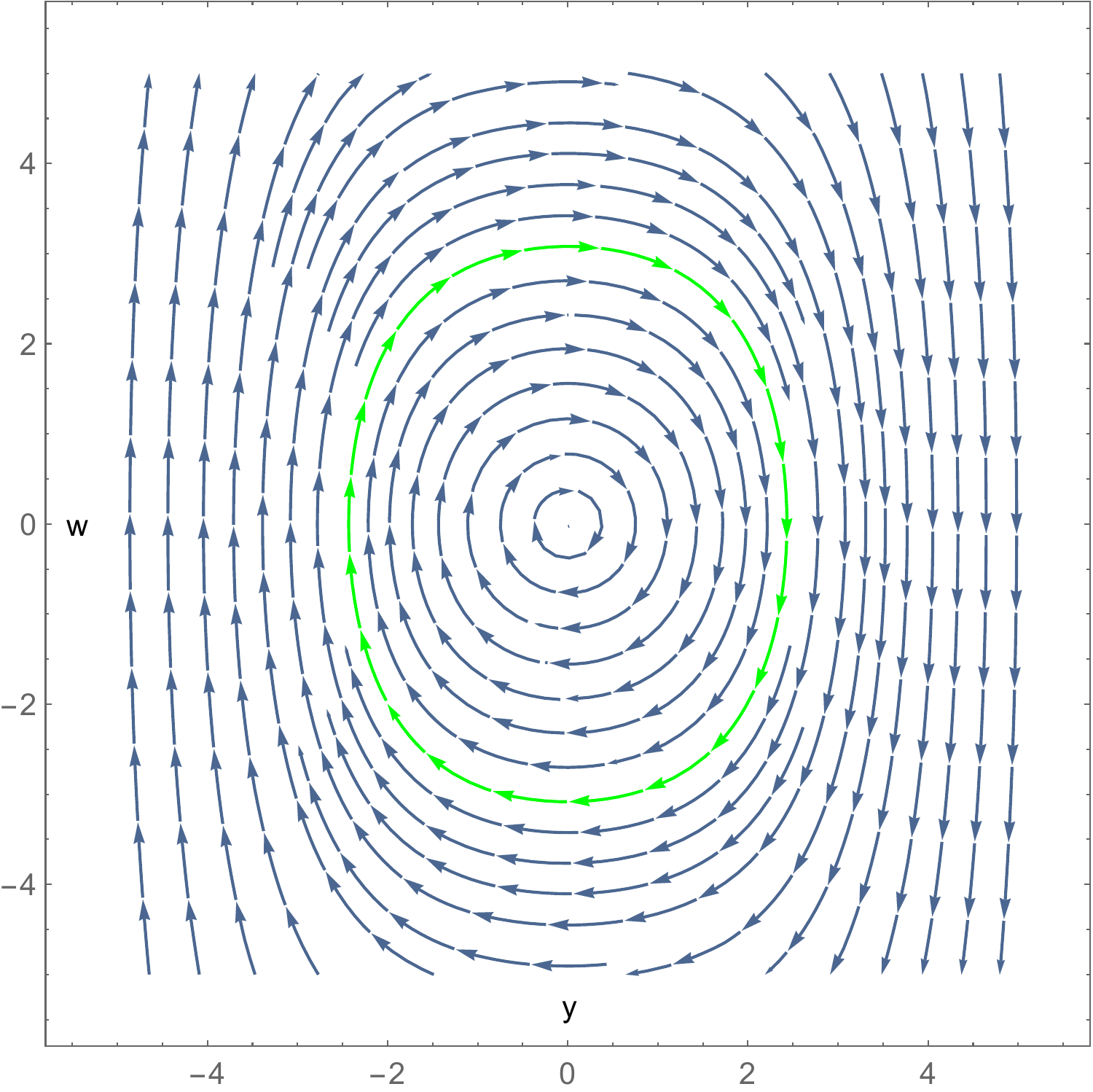}};

		\fill [white] (-0.5,-3.4) rectangle (0.5,-3.1);
		\fill [white] (-3.5,-1) rectangle (-3.1,1);
		
		\node at (-3.25,0.15) {$w$};
		\node at (0.1,-3.25) {$y$};

	    \node at (8.5,0)	{\includegraphics[scale=0.3, trim = 0.5cm 0.15cm 0.05cm 0.05cm, clip]{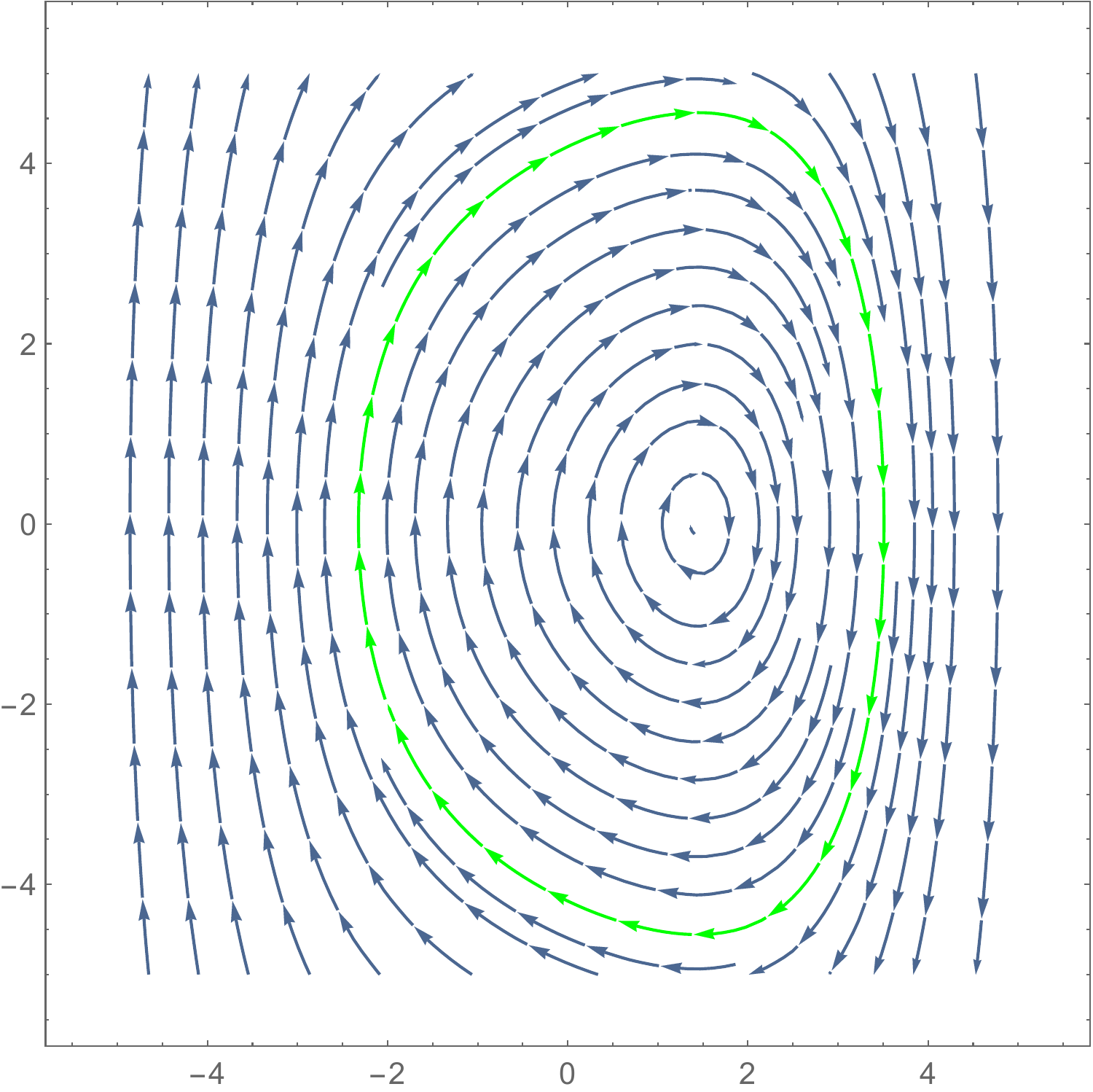}};

		\end{tikzpicture}
		\caption{Vector fields of \Cref{phaseportrait} with $d=1,\tilde{\kappa}=1$ with the left plot portraying $\lambda=0$, and the right $\lambda=2$.}
	\end{figure}

For $d \geq 2$, the vector field corresponding to \Cref{phaseportrait} is non-autonomous, and $F$ in \Cref{phaseportrait} is not well-defined at $r=0$ where our initial data are given. We regularize the ODE so that the regularized vector field is continuous (in $r$) near $r=0$, and show that the limiting solution satisfies \Cref{nonlinearBessel}. We solve an ODE that is slightly more general than \Cref{nonlinearBessel}. We use the notations of \Cref{nonunique}. 

\begin{lemma}\label{NonuniqueNonlinearBessel}
    For every $A \geq 0$ and $c\in \mathbb{R}$, there exists $R>\frac{C_B}{\tilde{\kappa}}$ and $y \in C^\infty_{\mathrm{loc}}((0,\infty);\mathbb{R})$ such that $y$ satisfies
\begin{equation}\label{NonuniqueNonlinearBessel2}
\begin{split}
    r y^{\prime\prime}+Ay^\prime + \tilde{\kappa}^2 r\sinh y&=r\lambda,\: r \in (0,\infty),\\
    \lim\limits_{r \to 0+} y(r) = c,\: \lim\limits_{r \to 0+} y^\prime(0+) = 0,\: y(R) &= \sinh^{-1}\Big(\frac{\lambda}{\tilde{\kappa}^2}\Big).
    \end{split}
\end{equation}
\end{lemma}

\begin{proof}[Proof of \Cref{nonunique}]
    Set $A = d-1$.
\end{proof}

\begin{proof}[Proof of \Cref{NonuniqueNonlinearBessel}]
    The $A=0$ case is equal to that when $d=1$, and therefore assume $A>0$. Moreover, assume $c \neq \sinh^{-1}\Big(\frac{\lambda}{\tilde{\kappa}^2}\Big)$ since it yields a trivial solution. For $\epsilon>0$, consider the perturbed ODE:
\begin{equation}\label{perturbedBessel}
\begin{split}
    (r+\epsilon)y_\epsilon^{\prime\prime} + Ay_\epsilon^\prime + \tilde{\kappa}^2(r+\epsilon) \sinh y_\epsilon &= (r+\epsilon)\lambda,\: r \in [-\frac{\epsilon}{2},\infty),\\
    y_\epsilon(0) = c,\: y_\epsilon^\prime(0) &= 0,
    \end{split}
\end{equation}
which, after setting $w_\epsilon = y_\epsilon^\prime$, reduces to

\[
   \begin{pmatrix}
			y_\epsilon\\
			w_\epsilon
		\end{pmatrix}^\prime = F_\epsilon(r,y_\epsilon,w_\epsilon)\coloneqq
		\begin{pmatrix}
			w_\epsilon\\
			-\tilde{\kappa}^2\sinh y_\epsilon - \frac{A w_\epsilon}{r+\epsilon} - \lambda
		\end{pmatrix}. 
\]

Since $F_\epsilon$ is smooth in $r$ near $r=0$ and locally Lipschitz in $(y,w)$, there exists $T_\epsilon \in (0, \frac{\epsilon}{2})$ and $y_\epsilon \in C([-T_\epsilon,T_\epsilon];\mathbb{R})\cap C_{\mathrm{loc}}^\infty((-T_\epsilon,T_\epsilon);\mathbb{R})$ such that $y_\epsilon$ is a unique solution to \Cref{perturbedBessel}. In the maximal interval of existence, 
$\begin{pmatrix}
y_\epsilon\\
w_\epsilon
\end{pmatrix}$
satisfies
\[
\frac{d}{dr} H(y_\epsilon(r),w_\epsilon(r)) = w_\epsilon(r)(y_\epsilon^{\prime\prime}(r) + \tilde{\kappa}^2\sinh y_\epsilon(r)-\lambda) = - \frac{Aw_\epsilon(r)^2}{r+\epsilon} \leq 0,\: r \geq -\frac{\epsilon}{2}
\]
and therefore the forward orbit of
$\begin{pmatrix}
y_\epsilon\\
w_\epsilon
\end{pmatrix}$
is bounded in the compact subset $\{(y,w) \in \mathbb{R}^2: H(y,w) \leq H(c,0)\}$ on which $F_\epsilon$ is Lipschitz. Hence,
$\begin{pmatrix}
y_\epsilon\\
w_\epsilon
\end{pmatrix}$
can be uniquely extended globally in forward time, obeying the estimate
\begin{equation}\label{uniformbound}
    H(y_\epsilon(r),w_\epsilon(r)) \leq H(c,0),\: r \geq 0.
\end{equation}

This global bound on $|y_\epsilon|+|w_\epsilon|$ yields an existence of a limit function, since for $r_1, r_2 \geq 0$,
\[
|y_\epsilon(r_2) - y_\epsilon(r_1)| = \left|\int_{r_1}^{r_2} w_\epsilon(\rho)d\rho\right| \leq C |r_2-r_1|,    
\]
where $C>0$ is independent of $\epsilon>0$. An immediate application of Arzel\`{a}-Ascoli Theorem implies that there exists a subsequence $\epsilon_k>0$ that tends to zero (from the right) and $y \in C_{\mathrm{loc}}([0,\infty);\mathbb{R})$ such that $y_{\epsilon_k}\xrightarrow[k\to 0]{}y$ in the topology of uniform convergence on compact subsets; in particular, $y(0)=c$.

Let $T>0$. Since $\{w_{\epsilon_k}\}$ is uniformly bounded in $L^2((0,T);\mathbb{R})$ due to \Cref{uniformbound}, there exists a subsequence of $\{\epsilon_k\}$ and $w \in L^2((0,T);\mathbb{R})$ such that, possibly after relabelling the subsequence, $w_{\epsilon_k} \rightharpoonup w$ in $L^2((0,T);\mathbb{R})$. This weak convergence of derivatives and the uniform convergence $y_{\epsilon_k} \rightarrow y$ on $[0,T]$ implies that $w$ is the weak derivative of $y$. Furthermore, we have $((r+\epsilon_k)y_{\epsilon_k}^\prime)^\prime (r) = (1-A) y_{\epsilon_k}^\prime - \tilde{\kappa}^2(r+\epsilon_k) \sinh y_{\epsilon_k}(r) + (r+\epsilon_k)\lambda$ from \Cref{perturbedBessel} where the right-hand side is uniformly bounded in $L^2((0,T);\mathbb{R})$. Another application of the Arzel\`{a}-Ascoli Theorem implies that there exists $Y \in C([0,T];\mathbb{R})$ such that $(\cdot+\epsilon_k)y_{\epsilon_k}^\prime \xrightarrow[k\to \infty]{} Y$ in $C([0,T];\mathbb{R})$, possibly after relabelling the subsequence, and follows $y_{\epsilon_k}^\prime \xrightarrow[k \to \infty]{} \frac{Y}{r}$ in $C([\delta,T];\mathbb{R})$ for every $\delta>0$, and therefore we identify $w(r)$ with a continuous function $\frac{Y(r)}{r}$ on $(0,T)$; indeed, $w = y^\prime$ classically on $(0,T)$. Yet another application of \Cref{uniformbound} and the triangle inequality $|w(r)| \leq |w(r)-y_{\epsilon_k}^\prime(r)| + |y_{\epsilon_k}^\prime(r)|$ yields the bound $|w(r)| \leq M$ for some $M>0$ on $(0,T)$. 

Since $y_{\epsilon_k}$ is a classical solution to \Cref{perturbedBessel}, it is also a weak solution. Writing \Cref{perturbedBessel} in the weak form, integrating by parts, and taking $k \rightarrow \infty$, we obtain $ry^{\prime\prime} + Ay^\prime +\tilde{\kappa}^2 r\sinh y = r\lambda$ on $(0,T)$ in the weak sense where the distributional derivative $y^{\prime\prime}$ can be identified with a continuous function on $(0,T)$ using the equation above. Using \Cref{perturbedBessel} and the uniform convergence of $y_{\epsilon_k}$ and its derivative as $k \rightarrow \infty$, we conclude $(r+\epsilon_k) y_{\epsilon_k}^{\prime\prime} \xrightarrow[k\to \infty]{} -(Ay^\prime + \tilde{\kappa}^2r \sinh y) + r\lambda = ry^{\prime\prime}$ uniformly on $[\delta,T]$, and therefore $y_{\epsilon_k}^{\prime\prime} \xrightarrow[k\to \infty]{}y^{\prime\prime}$ in $C([\delta,T];\mathbb{R})$ for every $\delta>0$. We have shown that $y_{\epsilon_k}^{(j)} \xrightarrow[k \to \infty]{} y^{(j)}$ uniformly on compact subsets of $(0,T)$ for $j=0,1,2$. Taking $k \to \infty$ from \Cref{perturbedBessel}, we conclude that $y$ satisfies the desired ODE pointwise on $(0,T)$. Since the vector field $F$ is smooth on $(0,T) \times \mathbb{R}^2$ where $T>0$ was arbitrary, we conclude $y \in C^\infty_{\mathrm{loc}}((0,\infty);\mathbb{R}))$.

Since $y^{\prime\prime}$ is continuous on $(0,T)$ and $y_{\epsilon_k}^\prime \rightharpoonup y^\prime$ in $L^2((0,T);\mathbb{R})$ as $k \to \infty$, for every $\phi \in C_c^\infty((0,T);\mathbb{R})$,
\[
    \int_0^T y_{\epsilon_k}^{\prime\prime}\phi = - \int_0^T y_{\epsilon_k}^\prime \phi^\prime \xrightarrow[k\to \infty]{} - \int_0^T y^\prime \phi^\prime = \int_0^T y^{\prime\prime}\phi.
\]
Therefore, $\{y_{\epsilon_k}^{\prime\prime}\}$ is uniformly bounded in $L^2((0,T);\mathbb{R})$, and another application of the Arzel\`{a}-Ascoli Theorem shows that there exists a convergent subsequence of $\{y_{\epsilon_k}^\prime\}$ in $C([0,T];\mathbb{R})$. Since we showed $y_{\epsilon_k}^\prime \xrightarrow[k\to \infty]{C([\delta,T];\mathbb{R})} y^\prime$ for every $\delta>0$, we conclude that $\delta$ could be taken to be zero. In particular, $y^\prime(0) = \lim\limits_{k \to \infty}y_{\epsilon_k}^\prime (0)=0$.

Finally from the phase portrait analysis, the solution $(y(r),w(r))$ exhibits an oscillatory behavior in $\mathbb{R}^2$; however, note that in the non-autonomous case, the solution curve does not lie in any curves of constant Hamiltonian due to the $y^\prime$ term. Since every closed curve of constant Hamiltonian contains the global minimum $(\sinh^{-1}\Big(\frac{\lambda}{\tilde{\kappa}^2}\Big),0)$, there exists $\{R_n\}_{n=1}^\infty$ such that $0<R_n < R_{n+1}, R_n \xrightarrow[n\to \infty]{}\infty$ such that $y(R_n)=0$.

\end{proof}

\begin{remark}
	For $d=3$, \Cref{nonlinearBessel} with $\tilde{\kappa}=1$ can be understood as a nonlinear zeroth-order spherical Bessel equation. To be more precise, a (linear) zeroth-order spherical Bessel ODE is given by
	\[
		r y^{\prime\prime} + 2 y^\prime + r y=0.    
	\]
	The two linearly independent solutions are given by
	\[
		j_0(r) = \frac{\sin r}{r};\: y_0(r) = -\frac{\cos r}{r}.
	\]
	We give plots comparing the linear and nonlinear solutions for $d=3$.
\end{remark}
	
	\begin{figure}[ht]
		\centering
\begin{tikzpicture}
        \footnotesize
     \node at (0,0) {\includegraphics[scale=0.30]{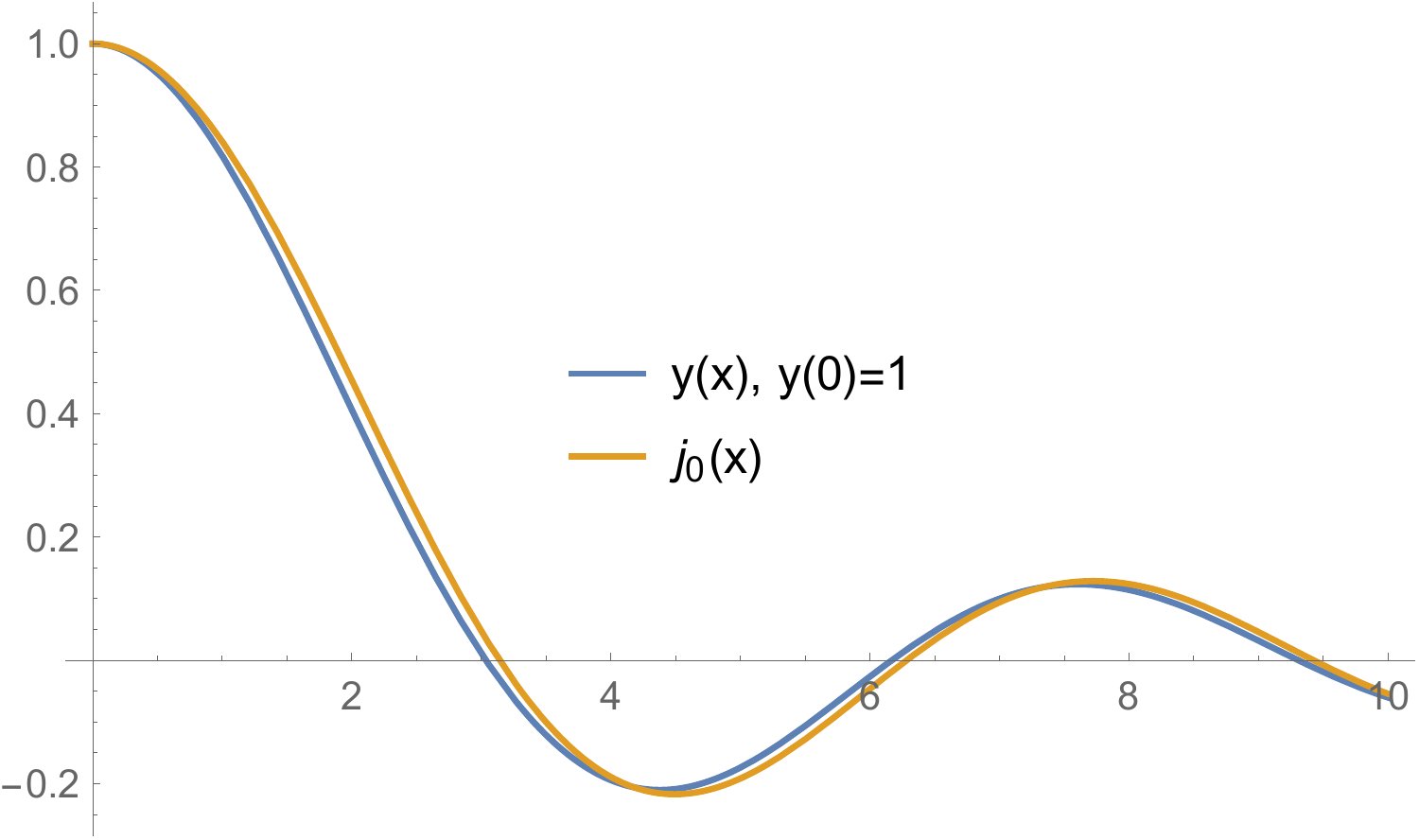}};
     \node at (8,0) {\includegraphics[scale=0.30]{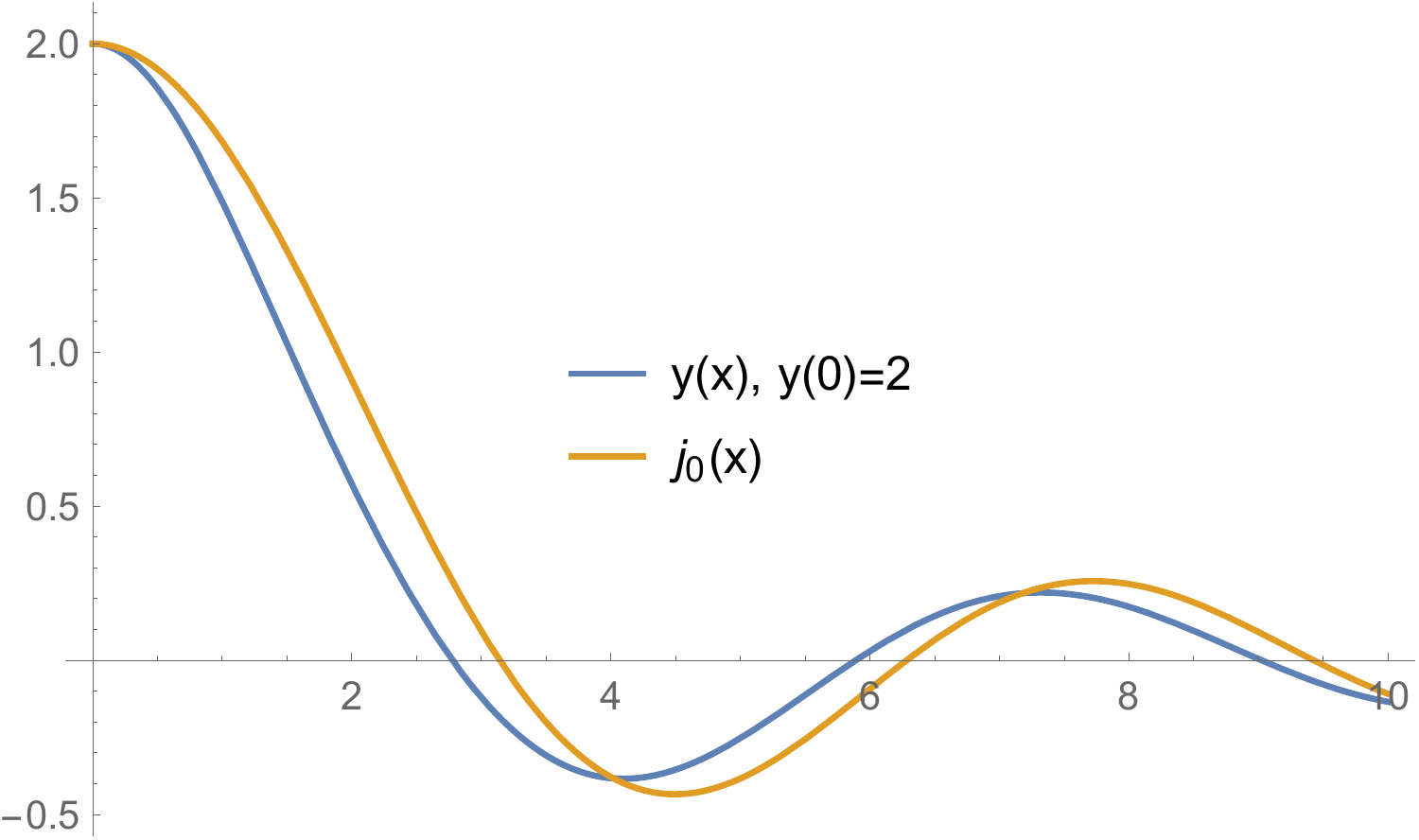}};
     \fill [white] (-0.3,-0.5) rectangle (1.25,0.5);
     \node at (0.7,0.25) {$y(x)$, $y(0)=1$};
     \node at (0.1,-0.2) {$j_0(x)$};
     \fill [white] (7.7,-0.5) rectangle (9.25,0.5);
     \node at (8.7,0.25) {$y(x)$, $y(0)=2$};
     \node at (8.1,-0.2) {$2j_0(x)$};
\end{tikzpicture}
		\caption{Comparison of solutions to \Cref{nonlinearBessel} and its corresponding linearization.}
	\end{figure}
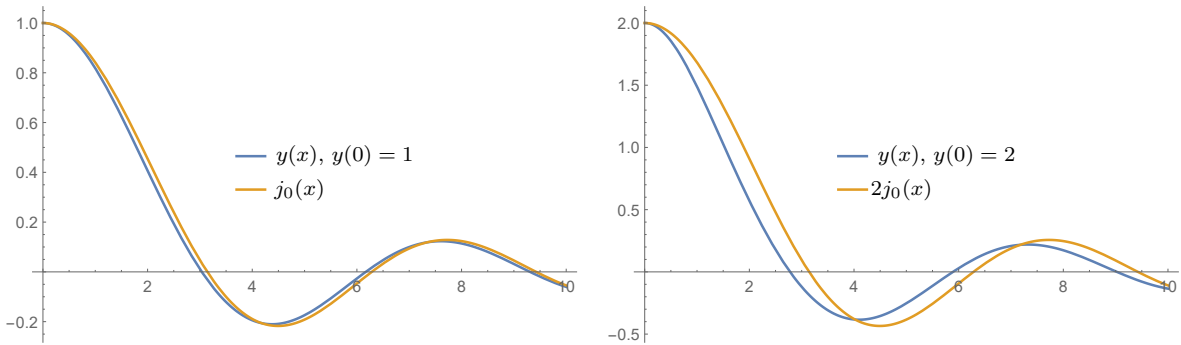

\begin{remark}
    Intuitively, this non-uniqueness stems from the non-coercivity of the nonlinear operator $Tu \coloneqq -\epsilon\Delta u + \kappa^2\sinh u$ when $\epsilon,\kappa$ do not satisfy $\epsilon,\kappa>0$. Indeed, our choice of nonlinearity is beyond the scope of those discussed in \cite{mcleod1993uniqueness} that studies the uniqueness of radial solution to $\Delta u + f(u)=0$ when $f^\prime(0)<0$. As a simple example, let $\epsilon=1,\kappa=i$ and consider the linearized equation $u^{\prime\prime} + u =0$ in $x \in (0,\pi)$ with the boundary condition $u(0)=u(\pi)=0$. Then, we have an uncountable family of solutions $\{A\sin x\}_{A \in \mathbb{R}}$. 
    
    However, it turns out that uniqueness can be salvaged if we drop the lower orders terms of $\sinh(u)$. We state a result whose proof, based on the Derrick-Pohozaev identity, is easily adapted from that of \cite[Section 9.4, Theorem 1]{evans2010partial}; compare this to \Cref{nonunique}.
    
    Let $d\geq 3$ and $N_0>\frac{d+2}{d-2}$ be an odd integer. Suppose $u \in C^2(\overline{\Omega})$ is a classical solution to
    	\[
    	\begin{split}
    	\label{counterexample2}
		-\Delta u &= \sum_{N \geq N_0,\: N\:odd} \frac{u^N}{N!},\: x \in \Omega\\
		u &=0,\: x \in \partial \Omega,\nonumber
	\end{split}
	\]
    where $\Omega$ is a star-shaped domain containing $0 \in \mathbb{R}^d$ with $\partial \Omega \in C^1$. Then, $u=0$ in $\overline{\Omega}$. 
\end{remark}
\end{appendices}

\noindent \textbf{Acknowledgements:} This material is based upon work supported by the
  National Science Foundation (NSF) under Grant No. 1736392.  Research
  reported in this technical report was supported in part by the
  National Institute of General Medical Sciences (NIGMS) of the
  National Institutes of Health under award number 1R01GM131409-03. The first author was partially funded by the NSF/RTG postdoctoral fellowship DMS-1840260.

\bibliographystyle{abbrv}
\bibliography{pdesref,citations,nPBEapplications,nPBEapplicationsGray}
\end{document}